\documentclass[11pt]{article} 

\usepackage[utf8]{inputenc} 

\usepackage{geometry} 
\usepackage{authblk}
\usepackage[all]{xy}
\usepackage[hidelinks]{hyperref}
\usepackage{graphicx} 
\usepackage{amsfonts,amsthm,amsmath,amssymb,esint}
 \usepackage{slashed}
\newcommand{\R}{\mathbb{R}}
\newcommand{\C}{\mathbb{C}}
\newcommand{\Z}{\mathbb{Z}}

\newcommand{\N}{\mathbb{N}}
\newcommand{\Sb}{\mathbb{S}}
\newcommand{\ket}{\rangle}
\newcommand{\bra}{\langle}
\newcommand{\T}{\mathbb{T}}
\newcommand{\Li}{\mathcal{L}}
\newcommand{\Ni}{\mathcal{N}}
\newcommand{\Ai}{\mathcal{A}}
\newcommand{\Bi}{\mathcal{B}}

\newcommand{\Hi}{\mathcal{H}}
\newcommand{\Ci}{\mathcal{C}}
\newcommand{\Qi}{\mathcal{Q}}

\newcommand{\Mi}{\mathcal{M}}
\newcommand{\Ki}{\mathcal{K}}

\newcommand{\Si}{\mathcal{S}}
\newcommand{\Ti}{\mathcal{T}}
\newcommand{\Ii}{\mathcal{I}}

\newcommand{\GA}{\mathfrak{A}}
\newcommand{\GH}{\mathfrak{H}}

\newcommand{\id}{\operatorname{id}}
\newcommand{\Aut}{\operatorname{Aut}}

\newcommand{\diag}{\operatorname{diag}}

\newcommand{\Ch}{\operatorname{Ch}}
\newcommand{\Tr}{\operatorname{Tr}}

\newcommand{\ep}{\varepsilon}
\newcommand{\Mn}{\operatorname{M}}
\newcommand{\Un}{\operatorname{U}}
\newcommand{\pd}{\partial}
\newcommand{\OP}{\operatorname{OP}}

\newcommand{\Ad}{\operatorname{Ad}}

\newtheorem{thm}{Theorem}[section]
\newtheorem{cor}[thm]{Corollary}

\newtheorem{Lemma}[thm]{Lemma}
\newtheorem{prop}[thm]{Proposition}

\theoremstyle{definition}
\newtheorem{Example}[thm]{Example}
\newtheorem{Notation}[thm]{Notation}
\theoremstyle{definition}

\newtheorem*{thm*}{Theorem}

\newtheorem*{Acknowledgment}{Acknowledgment}

\theoremstyle{definition}
\newtheorem{dfn}[thm]{Definition}
\newtheorem{Remark}[thm]{Remark}

\newcommand{\Index}{\operatorname{Index}}

\newcommand{\Ker}{\operatorname{Ker}}
\newcommand{\Dom}{\operatorname{Dom}}

\newcommand{\Sf}{\operatorname{Sf}}

\newcommand{\Ext}{\operatorname{Ext}}
\newcommand{\Res}{\operatorname{Res}}
\newcommand{\longto}{\longrightarrow}
\newcommand{\bone}{\mathbf{1}}

\usepackage{booktabs} 
\usepackage{array} 
\usepackage{paralist} 
\usepackage{verbatim} 
\usepackage{subfig} 

\usepackage{fancyhdr} 
\evensidemargin=\oddsidemargin
\usepackage{sectsty}
\allsectionsfont{\sffamily\mdseries\upshape} 

\usepackage[nottoc,notlof,notlot]{tocbibind} 
\usepackage[titles,subfigure]{tocloft} 

\title{Index pairings for $\R^n$-actions and Rieffel deformations}
\author{Andreas Andersson}
\usepackage{fancyhdr}
\affil{\small Email: fornjotnr@hotmail.com}
\affil{Wollongong University, School of Mathematics and Applied Statistics, 2522 Wollongong, Australia}
\affil{Mathematics Subject Classification 2010 Primary: 19K56; Secondary: 19K33, 19K35}
\affil{Keywords: Thom isomorphism, Kasparov $KK$-theory, spectral triples, noncommutative geometry
}

\date{August 10, 2016} 

\begin{document}
\maketitle
\abstract
With an action $\alpha$ of $\R^n$ on a $C^*$-algebra $A$ and a skew-symmetric $n\times n$ matrix $\Theta$ one can consider the Rieffel deformation $A_\Theta$ of $A$, which is a $C^*$-algebra generated by the $\alpha$-smooth elements of $A$ with a new multiplication. The purpose of this paper is to obtain explicit formulas for $K$-theoretical quantities defined by elements of $A_\Theta$. We give an explicit realization of Thom class in $KK$ in any dimension $n$, and use it in the index pairings. For local index formulas we assume that there is a densely defined trace on $A$, invariant under the action. When $n$ is odd, for example, we give a formula for the index of operators of the form $\slashed{P}\pi^\Theta(u)\slashed{P}$, where $\pi^\Theta(u)$ is the operator of left Rieffel multiplication by an invertible element $u$ over the unitization of $A$, and $\slashed{P}$ is projection onto the nonnegative eigenspace of a Dirac operator constructed from the action $\alpha$. The results are new also for the undeformed case $\Theta=0$. The construction relies on two approaches to Rieffel deformations in addition to Rieffel's original one: ``Kasprzak deformation" and ``warped convolution". We end by outlining potential applications in mathematical physics. 

\tableofcontents 

\section{Introduction}
The Thom isomorphism plays a fundamental role in $K$-theory of spaces \cite[\S2.7]{Atiy1}, \cite[\S IV.1]{Karo1}. In $K$-theory of $C^*$-algebras there is an analogous result, the ``Connes-Thom isomorphism", which plays an equally important role. 
\begin{thm}[{\cite[Thm. 2]{Co}}]\label{ConnesThomthm}
Let $(A,\R,\alpha)$ be a $C^*$-dynamical system. Then for $\bullet\in\Z_2=\{0,1\}$, there are natural isomorphisms
$$
\pd_\bullet:K_\bullet(A)\to K_{\bullet+1}(A\rtimes_\alpha\R).
$$
\end{thm}
It is a result of Fack and Skandalis \cite{FS} that the Connes-Thom isomorphism for one-parameter actions $\alpha:\R\to\Aut(A)$ is given by Kasparov product with a certain ``Thom class" $\mathbf{t}_\alpha$ in the Kasparov group $KK^1(A,B)$, where $B:= A\rtimes_\alpha\R$ is the crossed product. The class $\mathbf{t}_\alpha$ is a $KK$-equivalence with degree shift $1$, so the result is stronger than merely isomorphism in $K$-theory.

The direct proof of Theorem \ref{ConnesThomthm} given in \cite{Co} is \emph{not} available for $n\geq 2$. For some interesting remarks about why that is true, see \cite[\S 13]{Co9}. The $KK$ proof in \cite{FS} gives an explicit representative of the Thom element $\mathbf{t}_\alpha$. We would like to have such a representative of $\mathbf{t}_\alpha$ for actions by $\R^n$, with arbitrary $n$. 

Crossed products provide a rich source of examples to noncommutative index theory. By Theorem \ref{ConnesThomthm}, it is possible to use $B:=A\rtimes_\alpha\R$ to deduce $K$-theoretical information about $A$ and vice versa. Suppose that $A$ possesses a trace $\tau$ satisfying the invariance property $\tau\circ\alpha_t=\tau$ for all $t\in\R$. There is a canonical way of extending $\tau$ to a trace $\hat{\tau}$ on $B$, called the ``dual" of $\tau$. Now $\hat{\tau}$ induces a homomorphism
$$
\hat{\tau}_*:K_0(B)\to\R
$$ 
on $K$-theory. Suppose that $u$ is an $\alpha$-smooth unitary in the unitization $A^\sim=A\times\C$, defining a class $[u]$ in $K_1(A)$, and that $u$ belongs to the domain $\Dom(\tau)$ of the trace $\tau$. Then one has the formula \cite[Thm. 3]{Co}
\begin{equation}\label{Connformula}
\hat{\tau}(\pd_1([u]))=\frac{1}{2\pi i}\tau(u^*\delta(u)),
\end{equation}
where $\delta$ is the infinitesimal generator of the action $\alpha$. An explicit expression for the left-hand side of \eqref{Connformula} was then obtained in \cite{Le} by constructing a certain extension
\begin{equation}\label{ToeplLesch}
0\longto B\longto\Ti\longto A\longto 0
\end{equation}
of $A$ by the crossed product $B=A\rtimes_\alpha\R$. The ``Toeplitz algebra" $\Ti$ is generated by operators of the form 
$$
T_a=P\pi_\alpha(a)P,\qquad a\in A
$$
where $\pi_\alpha:A\to\Mi(B)$ is the usual embedding of $A$ into the multiplier algebra $\Mi(B)$ of the crossed product and $P$ is a projection in $\Mi(B)$. The short exact sequence \eqref{ToeplLesch} generalizes the classical Toeplitz extension of Coburn's for the algebra $A=C(\Sb^1)$ of continuous functions on the circle \cite{Cob1, Cob2}. The formula derived in \cite{Le} reads
\begin{equation}\label{Leschform}
\Index_{\hat{\tau}}(T_u)=-\frac{1}{2\pi i}\tau(u^*\delta(u)),
\end{equation}
where, for an operator $T$ whose kernel and cokernel projections $\Ker(T)$ and $\Ker(T^*)$ belong to the domain of $\hat{\tau}$,
$$
\Index_{\hat{\tau}}(T):=\hat{\tau}(\Ker(T))-\hat{\tau}(\Ker(T^*))
$$
is the $\hat{\tau}$-Fredholm index $T$. If $\hat{\tau}$ is the operator trace $\Tr:\Bi(\Hi)_+\to[0,+\infty]$ then $\Index_{\hat{\tau}}(T)$ is the ordinary Fredholm index. In \cite{Le}, the trace $\tau$ is assumed finite and $A$ is assumed unital, but the formula \eqref{Leschform} has been extended to nonunital $A$ and densely defined $\tau$ \cite{PR}. Then \eqref{Leschform} is a generalization of the classical Gohberg-Krein theorem \cite[\S 10]{GK}. 

Replacing $A$ by $\Mn_r(A)$ for any $r\in\N$ and tensoring $\tau$ with the standard trace on $\Mn_r(\C)$ one obtains a similar Toeplitz extension and the same formula \eqref{Leschform} for $u\in \Un_r(A^\sim)$. Thus we know from \cite{Le, PR} that for $n=1$ we obtain a map from the odd $K$-theory of $A$ to the even $K$-theory of $B$. The map 
$$
\partial_1:K_1(A)\to K_0(B),\qquad \partial_1([u]):=\Index(T_u),
$$
where $\Index(T):=[\Ker(T)]-[\Ker(T^*)]$ is the abstract index in $K_0(B)$, is an explicit description of Connes' Thom isomorphism in odd dimensions. 


We shall attempt a generalization of these results to higher dimensions $n$, i.e. to $C^*$-dynamical systems $(A,\R^n,\alpha)$. We shall realize the Thom isomorphism as the operation of right Kasparov product by a class $\mathbf{t}_\alpha\in KK^\bullet(A,B)$ both in even and odd dimensions $n$, where $\bullet\in\{0,1\}=\{\rm even,\rm odd\}$ is the parity of the integer $n$, and we find explicit bounded and unbounded representatives of $\mathbf{t}_\alpha$. Using the unbounded representative of $\mathbf{t}_\alpha$ we generalize the formula \eqref{Leschform} to arbitrary separable $C^*$-algebras $A$ equipped with a strongly continuous action $\alpha$ and a faithful densely defined $\alpha$-invariant lower-semicontinuous trace $\tau$. The precise statements will be given in \S\ref{stateofressec}. It may be thought of as a ``noncommutative Gohberg-Krein theorem".

In order to show that a ``local" formula such as \eqref{Leschform} is available for \emph{each} class $[u]\in K_1(A)$ we must find a dense $*$-subalgebra $\Ci$ of $A$ such that the inclusion $\Ci\hookrightarrow A$ induces an isomorphism on $K$-theory (briefly, we need to find a ``local subalgebra" of $A$) and such that the formula \eqref{Leschform} holds for all unitaries $u$ over the unitization $\Ci^\sim$. That is a highly nontrivial task, and at several places one needs to apply tools which were developed quite recently \cite{CPRS1, CPRS2, CPRS3, CPRS4, CGRS1, CGRS2}. After much work we will find a ``smoothly summable" spectral triple $(\Ci,\Hi,\slashed{D})$ such that, in both even and odd dimensions $n$, a generalization of \eqref{Leschform} works for matrices over $\Ci^\sim$. Even then one has to resort to \cite[Prop. 2.20]{CGRS1} before one can conclude that the local formula works for all $K$-classes. Thus, an explicit realization of the Thom isomorphism for $\R^n$-actions relies on many aspects of the general noncommutative index theory formulated in \cite{CGRS1}.

Let us now mention our main motivation for obtaining a local formula for the index pairing with the Thom class. If $\alpha:\R^n\to\Aut(A)$ is a strongly continuous action of $\R^n$ on a $C^*$-algebra $A$ and $\Theta$ is a real skew-symmetric $n\times n$-matrix, one can introduce a new $\Theta$-dependent multiplication $\times_\Theta$ on the $\alpha$-smooth subalgebra $\Ai$ of $A$. The norm completion of $\Ai$ in a ``left regular representation" of $(\Ai,\times_\Theta)$ is then a $C^*$-algebra $A_\Theta$, called the ``Rieffel deformation" of $A$ with respect to the data $(\alpha,\Theta)$ \cite{Rie1}. An alternative approach to Rieffel deformation using crossed products was introduced by Kasprzak \cite{Kas}. Sparing the details until \S\ref{deform}, the Kasprzak approach suggests that if we can do index pairings via crossed products then we might be able to do index pairings for Rieffel deformations as well, and to compare the Thom elements of $A$ and $A_\Theta$. For a comparison of the numerical deformed and undeformed indices we need a third approach to Rieffel deformations called ``warped convolutions" \cite{BLS}. These allow the deformed algebra $A_\Theta$ to be represented on any Hilbert space on which $A$ is represented. In fact, the applications of Rieffel deformations that we have in mind appear in the form of warped convolutions \cite{An1, An2, BV, L, Mu1, Mu2, V}. Our main result, Therem \ref{headthmRiefl}, is a local formula in terms of warped convolutions for the index pairings associated with the deformed $C^*$-dynamical system $(A_\Theta,\R^n,\alpha^\Theta)$. 

\begin{Acknowledgment} 
The author thanks Adam Rennie for many discussions, support and help. We also thank Rainer Verch, Gandalf Lechner and Andreas Thom for various comments. This paper was written when the author was still affiliated with the Max Planck Institute for Mathematics in the Sciences.
\end{Acknowledgment}
More detailed background to all the material used in this paper can be found in \cite{An3}.

\section{Index pairings for $\R^n$-actions}

\subsection{Notation and statement of the results}\label{stateofressec}

Let $A$ be a separable $C^*$-algebra and let $\alpha$ be a strongly continuous action of $\R^n$ on $A$. Briefly, we say that $(A,\R^n,\alpha)$ is a $C^*$-\textbf{dynamical system}. Let $A$ be identified with its image $A\subset\Bi(\GH)$ in some a faithful representation. We assume that $\alpha$ is unitarily implemented in $\GH$ (for example, this always happens if the weak closure $A''$ is in standard form \cite[Chapter IX.1]{Ta2}). 
The Hilbert space $\GH$ determines a representation
$$
\pi_\alpha:A\to\Bi(L^2(\R^n,\GH))
$$
where, for all $a\in A$ and $\xi\in L^2(\R^n,\GH)$,
$$
(\pi_\alpha(a)\xi)(t):=\alpha_{-t}(a)\xi(t),\qquad\forall\, t\in\R^n.
$$
 Let $D_1,\dots,D_n$ be generators of the $n$-parameter group of unitaries on $L^2(\R^n,\GH)$ implementing $\alpha$ in $L^2(\R^n,\GH)$, i.e. 
$$
\pi_\alpha(\alpha_t(a))=e^{2\pi it\cdot D}\pi_\alpha(a)e^{-2\pi it\cdot D},\qquad\forall a\in A,\ t\in\R^n,
$$
where $t\cdot D:=t_1D_1+\cdots+t_nD_n$. 
\begin{dfn}
Equip the Banach space $L^1(\R^n,A)$ with the convolution product
$$
(f*g)(t):=\int_{\R^n}f(s)\alpha_s(g(t-s))\, ds,
$$
and the involution $f^*(t):=\alpha_{t}(f(-t))^*$. The \textbf{crossed product} of $A$ by $\R^n$ is the $C^*$-algebra $A\rtimes_\alpha\R^n$ generated by the image of $L^1(\R^n,A)$ in the representation $\tilde{\pi}_\alpha:L^1(\R^n,A)\to \Bi(L^2(\R^n,\GH))$ given by \cite[Prop. 2.39]{Will1}
$$
\tilde{\pi}_\alpha(f):=\int_{\R^n}\pi_\alpha(f(t))e^{-2\pi i t\cdot D}\, dt,\qquad\forall f\in L^1(\R^n,A).
$$
\end{dfn}
We thus regard the crossed product $B:=A\rtimes_\alpha\R^n$ as a concrete $C^*$-algebra of operators on $L^2(\R^n,\GH)$. The isomorphism class of $B$ is independent of the choice of Hilbert space $\GH$ in which $\alpha$ is unitarily implemented \cite[\S7.2]{Will1}. Also the von Neumann algebra $\Ni:=B''$ is independent of the choice of $\GH$, up to isomorphism \cite[Thm. X.1.7]{Ta2}, and we fix such an $\GH$ and the corresponding $\pi_\alpha$. 

So the natural representation of the crossed product is on the Hilbert space $L^2(\R^n,\GH)=L^2(\R^n)\otimes\GH$. However, in order to obtain a representative of the Thom class, we shall need to consider the Hilbert space
$$
\Hi:=\C^{N}\otimes L^2(\R^n,\GH),
$$
where $\C^N$ carries an irreducible representation of the $n$-dimensional complex Clifford algebra $\C_n$. Explicitly, 
\[ N:= \left\{ 
  \begin{array}{l l}
    2^{n/2} & \quad \text{if $n$ is even,}\\
    2^{(n-1)/2} & \quad \text{if $n$ is odd.}
  \end{array} \right.\]
The $n$-dimensional complex Clifford algebra $\C_n$ can then be identified with
\[\C_n\cong \left\{ 
  \begin{array}{l l}
    \Mn_N(\C) & \quad \text{if $n$ is even,}\\
    \Mn_N(\C)\oplus\Mn_N(\C) & \quad \text{if $n$ is odd.}
  \end{array} \right.\]
The irreducible representation of $\C_n$ for even $n$ is on $\C^N$. For odd $n$ there are two irreducible representations, given by sending the first respectively the second $\Mn_N(\C)$-summand in $\C_n$ to the fundamental representation of $\Mn_N(\C)$ on $\C^N$.

The representation of $A$ on $\Hi$ is the diagonal one,
$$
\pi_\alpha(a):=1_{N}\otimes \pi_\alpha(a),
$$
where $1_N$ is the identity matrix of size $N\times N$. The selfadjoint operators $D_1,\dots,D_n$ can be used to define the Dirac operator (the tensor product implicit)
\begin{equation}\label{Diracoperator}
\slashed{D}:=\sum^n_{k=1}\gamma^kD_k
\end{equation}
in $\Hi$, where $\gamma^1,\dots,\gamma^n$ are Hermitian $N\times N$ matrices representing the generators of $\C_n$ on $\C^N$, satisfying therefore the Clifford relations $\gamma^j\gamma^k+\gamma^k\gamma^j=2\delta^{j,k}$.

Since $\slashed{D}$ is not invertible, yet one more ``doubling-up" trick is necessary for the upcoming considerations. Consider the Hilbert space $\boldsymbol\Hi:=\Hi\otimes\C^2$ and the operator
\begin{equation}\label{massiveDirackasp}
\slashed{\boldsymbol D}:=\begin{pmatrix}
\slashed{D}&0\\
0&-\slashed{D}
\end{pmatrix}+m\begin{pmatrix}
0&\bone\\
\bone&0
\end{pmatrix}
\end{equation}
for some arbitrary $m>0$. We let $\slashed{\boldsymbol P}$ denote the spectral projection of $\slashed{\boldsymbol D}$ corresponding to the interval $[0,+\infty)$. If $\slashed{\boldsymbol R}:=\slashed{\boldsymbol D}|\slashed{\boldsymbol D}|^{-1}$ is the ``phase" of $\slashed{\boldsymbol D}$ then we have $\slashed{\boldsymbol P}=(\bone+\slashed{\boldsymbol R})/2$. 
We represent the minimal unitization $A^\sim=A\times\C$ on $\boldsymbol\Hi$ by setting
$$
\boldsymbol\pi_\alpha(a+\lambda\bone):=\begin{pmatrix}\pi_\alpha(a)+\lambda\bone&0\\0&\lambda\bone\end{pmatrix}
$$
for $a\in A$ and $\lambda\in\C$. 
\begin{dfn}
The \textbf{Toeplitz algebra} of $(A,\R^n,\alpha)$ is the $C^*$-subalgebra $\boldsymbol\Ti$ of $\Bi(\Hi)$ generated by $\Mn_N(\boldsymbol B)$ together with elements of the form 
$$
T_a:=\slashed{\boldsymbol{P}}\boldsymbol\pi_\alpha(a)\slashed{\boldsymbol{P}}
$$ 
for $a\in A$.
\end{dfn}
Our first result is the following Toeplitz extension.
\begin{prop}\label{exactsequence} There is a semisplit short exact sequence
$$
0\longto\Mn_N(\boldsymbol B)\longto\boldsymbol\Ti \longto A \longto 0.
$$
\end{prop}
In a sense that will be made precise below, the triple $(\boldsymbol{\pi}_\alpha,\Mn_N(\boldsymbol B),\slashed{\boldsymbol R})$ carries the same $K$-theoretical information as the triple $(\pi_\alpha,\Mn_N(B),\slashed{F})$, where $\slashed{F}:=\slashed{D}(1+\slashed{D}^2)^{-1/2}$.

For even $n$, the operator $\Gamma:=(-i)^{n/2}\gamma^1\cdots\gamma^n$ gives a $\Z_2$-grading of $\Hi$, which we write as $\Hi=\Hi_+\oplus\Hi_-$, and we have $[\Gamma,\pi_\alpha(a)]=0$ for all $a\in A$ while $\Gamma\slashed{D}=-\slashed{D}\Gamma$. For instance, consider the explicit expressions for the Dirac operator in low dimensions $n=1,2,3$ given by 
$$
\slashed{D}=-D_1 \qquad \text{if $n=1$},
$$
$$
\slashed{D}=\begin{pmatrix}0&iD_1+D_2\\-iD_1+D_2&0\end{pmatrix}\qquad \text{if $n=2$},
$$
$$
\slashed{D}=\begin{pmatrix}D_3&iD_1+D_2\\-iD_1+D_2&-D_3\end{pmatrix}\qquad \text{if $n=3$}.
$$
For both $n=2$ and $n=3$, the $\gamma^k$'s in \eqref{Diracoperator} are the Pauli matrices. For $n=2$ we can take $\Gamma=\gamma^3\otimes\bone$, with the third Pauli matrix $\gamma^3=\diag(1,-1)$ (the diagonal matrix with eigenvalues $1$ and $-1$) and obtain a grading in which $\pi(A)$ is even and $\slashed{D}$ is odd. In contrast, for $n=3$ the diagonal terms in $\slashed{D}$ spoil the property $\Gamma\slashed{D}=-\slashed{D}\Gamma$.

The phase $\slashed{\boldsymbol{R}}=\slashed{\boldsymbol{D}}|\slashed{\boldsymbol{D}}|^{-1}$ decomposes in $\boldsymbol\Hi=\boldsymbol\Hi_+\oplus\boldsymbol\Hi_-$ as
$$
\slashed{\boldsymbol{R}}=\begin{pmatrix}0&\slashed{\boldsymbol{R}}_-\\\slashed{\boldsymbol{R}}_+&0\end{pmatrix}
$$
with $\slashed{\boldsymbol{R}}_-=(\slashed{\boldsymbol{R}}_+)^*$.

\begin{thm}\label{Thomthm} 
Let $(A,\R^n,\alpha)$ be a $C^*$-dynamical system, with $A$ separable. Let $B:=A\rtimes_\alpha\R^n$ be the crossed product and let $\bullet\in\{0,1\}=\{\rm even,\rm odd\}$ be the parity of the integer $n$. The Thom class $\mathbf{t}_\alpha\in KK^\bullet(A,B)$ is represented by any of the following Kasparov $A$-$B$ modules.
\begin{enumerate}[(i)]
\item{$(\pi_B,\Mn_N(B),\slashed{F})$, where $\slashed{F}:=\slashed{D}(\bone+\slashed{D}^2)^{-1/2}$ is the bounded transform of $\slashed{D}$, and the representation $\pi_B:A\to\Mi(B\otimes\Ki)$ is given by left multiplication via $\pi_\alpha$.}
\item{$(\boldsymbol\pi_B,\Mn_N(\boldsymbol B),\slashed{\boldsymbol R})$.}
\item{$(\pi_B,\Mn_N(B),2\slashed{P}-\bone)$, where $\slashed{P}$ is the spectral projection of the Dirac operator $\slashed{D}$ corresponding to the interval $[0,+\infty)$.}
\end{enumerate}
If $n$ is odd then the Thom class $\mathbf{t}_\alpha$ is also represented by the Toeplitz extension from Proposition \ref{exactsequence}, under the identification of $KK^1(A,B)$ with $\Ext(A,B)^{-1}$ (the group of semisplit extensions of $A$ by $B\otimes\Ki$). 
\end{thm}

We now want a more explicit formula for the index pairing. For $k=1,\dots, n$, let $\delta_k$ denote the infinitesimal generator of $\alpha$ in the $k$th direction, so that $\pi_\alpha(\delta_k(a))=2\pi i[D_k,\pi_\alpha(a)]$. The common smooth domain of $\delta_1,\dots,\delta_n$ will be denoted by $\Ai$. For $a,b\in\Ai$ and $m=1,\dots,n$ we use the shorthand notation
\begin{equation}\label{shrthand}
(a\delta(b))^m:=\sum_{\ep}(-1)^\ep\prod^m_{k=1} a\delta_{\ep(k)}(b),
\end{equation}
where the sum is over all permutations $\ep$ of $\{1,\dots,n\}$. 

In order to go beyond abstract index theory and be able to talk about real-valued indices, we require that the $C^*$-algebra $A$ admits a densely defined trace $\tau:A_+\to[0,+\infty]$ (here and below $A_+$ denotes the positive cone in a $*$-algebra $A$). Then there exists a weight $\hat{\tau}$ on $B:=A\rtimes_\alpha\R^n$ such that
\begin{equation}\label{Cstardualtr}
\hat{\tau}(\hat{\pi}_\alpha(f)^*\hat{\pi}_\alpha(f))=\tau(\bra f|f\ket_A),\qquad\forall f\in C_0(\R^n,\Dom(\tau))\cap L^2(\R,\GH)
\end{equation}
where $\bra \cdot|\cdot\ket_A$ is the $A$-valued inner product given by
\begin{equation}\label{Cstarinnerproduct}
\bra f|g\ket_A:=\int_{\R^n}f(s)^*g(s)\, ds, \qquad \forall f,g\in L^2(\R^n,A).
\end{equation}
For clarity we denote by $\bar{\tau}$ the normal extension of $\tau$ to the von Neumann algebra $A''$. Then we have the more well-established $W^*$-notion dual weight of $\bar{\tau}$ on the $W^*$-crossed product $\Ni:=A''\rtimes_\alpha\R^n$, which extends $\hat{\tau}:\Mi_+\to[0,\infty]$. We write $\hat{\tau}$ also for this extension. Let us make the definition more precise. 
\begin{dfn}[{\cite[Def. X.1.16]{Ta2}, \cite{H1,H2}}]\label{dualtrace} 
The weight on $\Ni$ \textbf{dual} to $\tau$ is defined to be
\begin{equation}\label{defdualweight}
\hat{\tau}:=\tau\circ\pi_\alpha^{-1}\circ E,
\end{equation}
where $E$ is the ``operator-valued weight" (see \cite[Def. 2.1]{H3}) from $\Ni$ to the fixed-point subalgebra $\Ni^{\hat{\alpha}}=\pi_\alpha(\Mi)$ given by (here $\hat{\alpha}$ is the dual action)
\begin{equation}\label{opvaluedweight}
E(T^*T):=\int_{\R^n}\hat{\alpha}_p(T^*T)\, dp, \qquad \forall T\in\Ni.
\end{equation}
Here $E(T^*T)$ makes sense as an ultraweakly lower semicontinuous map from the positive cone in the predual of $\pi_\alpha(\Mi)$ into $[0,+\infty]$. 
\end{dfn}
\begin{Remark}\label{modularautdualtrace}
In our generalization of the Gohberg-Krein theorem we shall need to assume $\tau$ to be invariant under the $\R^n$-action. To see why, suppose that $\tau(\alpha_t(a))=\tau(\rho^{it}a)$ for all $a\in A$, $t\in\R$ for some positive invertible operator $\rho$ affiliated to $A''$. Then the modular automorphism group $\sigma^{\hat{\tau}}$ of $\hat{\tau}$ is nontrivial, namely 
$$
\sigma^{\hat{\tau}}_t(x)=x,\qquad \sigma^{\hat{\tau}}_t(e^{2\pi is\cdot D})=\pi_\alpha(\rho^{it})e^{2\pi is\cdot D},\qquad\forall x\in\Ni,\,s,t\in\R^n.
$$
So $\hat{\tau}$ is not a trace in this case. On the other hand, if $\tau$ is $\alpha$-invariant then $\sigma^{\hat{\tau}}_t\equiv \id$, which is equivalent to saying that $\hat{\tau}$ is a trace.
\end{Remark}
Thus, we shall assume that the trace $\tau$ is $\alpha$-invariant. We shall also assume that $\tau$ is faithful and lower semicontinuous in norm. Then the dual weight $\hat{\tau}$ is a normal faithful semifinite trace on $\Ni$, so that semifinite Fredholm theory \cite{Az1, BeFa1, Breu1, Breu2, CGRS1} is available.

We will find a $*$-subalgebra $\Ci$ of $\Ai$ of elements which are both sufficiently ``smooth" with respect to $\slashed{D}$ and ``integrable" with respect to $(\slashed{D},\hat{\tau})$. It is from elements of this algebra that our $K$-theoretical quantities can be explicitly calculated. 

For notation simplicity we will formulate the result for unitaries $u$ and projections $e$ in $\Ci^\sim$. It is easily adapted to matrices over $\Ci^\sim$ as well. 
\begin{thm}\label{headthm} 
Let $(A,\R^n,\alpha)$ be a $C^*$-dynamical system, with $A$ separable, and suppose that $\tau$ is a faithful densely defined lower semicontiuous $\alpha$-invariant trace on $A$. There exists a local subalgebra $\Ci$ of $A$ such that $(\Ci,\Hi,\slashed{D})$ is a smoothly summable spectral triple over $A$, with spectral dimension $n$. 

Suppose that $n$ is odd. For each unitary $u\in \Ci^\sim$, the $\hat{\tau}$-index of the Toeplitz operator $\slashed{{P}}\pi_\alpha(u)\slashed{{P}}$ can be calculated using the the ``local" formula
\begin{align*}
\Index_{\hat{\tau}}(\slashed{{P}}\pi_\alpha(u)\slashed{{P}})&=
-\frac{2^{(n-1)/2}(-1)^{(n-1)/2}((n-1)/2)!}{(2\pi i)^nn!}\tau\big((u^*\delta(u))^n\big)
\end{align*}
where we use the notation introduced in \eqref{shrthand}. Suppose that $n$ is even. Then for each projection $e\in\Ci^\sim$, one has 
\begin{align*}
\Index_{\hat{\tau}}(\boldsymbol\pi_\alpha(e)\slashed{\boldsymbol{R}}_+\boldsymbol\pi_\alpha(e))&=
\frac{(-1)^{n/2}}{(n/2)!}\frac{2^n}{(2\pi i)^n}\tau\big((e\delta(e)\delta(e))^{n/2}\big),
\end{align*}
where $\slashed{\boldsymbol{R}}_+:\Hi_+\to\Hi_-$ is the $+$-part of $\slashed{\boldsymbol{R}}=\slashed{\boldsymbol{D}}|\slashed{\boldsymbol{D}}|^{-1}$ under the splitting $\Hi=\Hi_+\oplus\Hi_-$.   
\end{thm}
Theorem \ref{headthm} gives a generalization of the $n=1$ formulae in \cite{CGPRS, Le, PR}. If furthermore  $A=C(\Sb^1)$, $\Ci=C^\infty(\Sb^1)$, with $\slashed{D}=\sqrt{-1}\partial/\partial t$ and $\tau$ the Lebesgue integral, we get the classical Gohberg-Krein theorem (see \cite[\S4(a)]{PR}).

As a consequence of the general theory in \cite{CGRS1} we also obtain the following facts (which we will not dwell more upon).
\begin{cor}\label{extracor}
Suppose that $n$ is odd and let $u\in \Ci^\sim$ be a unitary. The $\hat{\tau}$-index of the Toeplitz operator $\slashed{{P}}\pi_\alpha(u)\slashed{{P}}$ is equal to the spectral flow (see \cite{BCPRSW, CPRS1, CPS, Getz1, KNR, Ph1}) between $\slashed{\boldsymbol{D}}$ and $\boldsymbol\pi_\alpha(u^*)\slashed{\boldsymbol{D}}\boldsymbol\pi_\alpha(u)$,
\begin{align*}
\Index_{\hat{\tau}}(\slashed{{P}}\pi_\alpha(u)\slashed{{P}})
&=\Sf(\slashed{\boldsymbol{D}},u^*\slashed{\boldsymbol{D}}u),
\end{align*}
and to the pairing between the Chern character $\Ch(u)\in HP_1(\Ci)$ in odd continuous periodic cyclic homology with the cohomological Chern character $\Ch(\Ai,\Hi,\slashed{D})\in HP^1(\Ci)$,
\begin{align}
\Index_{\hat{\tau}}(\slashed{{P}}\pi_\alpha(u)\slashed{{P}})
&=\frac{-1}{\sqrt{2\pi i}}\bra\Ch(u),\Ch(\Ci,\Hi,\slashed{D})\ket\label{oddThomis}
\\&=-\hat{\tau}\big(\slashed{P}[\slashed{P},\pi_\alpha(u^{-1})][\slashed{P},\pi_\alpha(u)]\cdots[\slashed{P},\pi_\alpha(u^{-1})][\slashed{P},\pi_\alpha(u)]\big)\nonumber.
\end{align}
Similarly, if $n$ is even and $e\in\Ci^\sim$ is a projection, the $\hat{\tau}$-index of $\boldsymbol\pi_\alpha(e)\slashed{\boldsymbol{R}}_+\boldsymbol\pi_\alpha(e)$ is the pairing between the Chern character $\Ch(e)\in HP_0(\Ci)$ in even continuous periodic cyclic homology with the cohomological Chern character $\Ch(\Ai,\Hi,\slashed{D})\in HP^0(\Ci)$,
\begin{align}\label{evenThomis}
\Index_{\hat{\tau}}(\boldsymbol\pi_\alpha(e)\slashed{\boldsymbol{R}}_+\boldsymbol\pi_\alpha(e))&=\bra\Ch(e),\Ch(\Ci,\Hi,\slashed{D})\ket.
\end{align}
\end{cor}
In the sense of \eqref{oddThomis} and \eqref{evenThomis} we may say that $(\Ci,\Hi,\slashed{D})$ is an ``unbounded representative" of the Thom class for $(A,\R^n,\alpha)$. Indeed, the left-hand sides of \eqref{oddThomis} and \eqref{evenThomis} are obtained by applying the homomorphism $\hat{\tau}_*:K_0(B)\to\R$ to the Kasparov products $[u]\otimes_A\mathbf{t}_\alpha$ and $[e]\otimes_A\mathbf{t}_\alpha$ (see Corollary \ref{oddandevenFredpair}).

Corollary \ref{extracor} is a direct consequence of the (nontrivial) fact that $(\Ci,\Hi,\slashed{D})$ is a smoothly summable spectral triple over $A$. The local formula in Theorem \ref{headthm} relies on properties of the dual trace $\hat{\tau}$ which allow us to translate certain expressions back to the original trace $\tau$ on $A$.

\subsection{The Dirac operator}\label{Diracsection}

\begin{prop}\label{easycommutator} For $a$ in the intersection of the domains $\Dom(\delta_k)$ of the generators $\delta_k$ of $\alpha$, we have
$$
[\slashed{D},\pi_\alpha(a)]=\frac{1}{2\pi i}\sum^n_{k=1}\gamma^k\pi_\alpha(\delta_k(a)).
$$
\end{prop}
\begin{proof}This is seen as in \cite[Prop. 3.3]{CGPRS} using
$$
[\slashed{D},\pi_\alpha(a)]=\sum^n_{k=1}\gamma^k[D_k,\pi_\alpha(a)].
$$
Namely, if $\xi$ is in the domain of $D_k$ and $a$ is in the domain of $\delta_k$ then
\begin{align*}
(D_k\pi_\alpha(a)\xi)(t)&=\frac{1}{2\pi i}\frac{\partial}{\partial t_k}(\alpha_{-t}(a)\xi(t))
\\&=\frac{1}{2\pi i}\alpha_{-t}(\delta_k(a))\xi(t)+\frac{1}{2\pi i}\alpha_{-t}(a)\frac{\partial}{\partial t_k}\xi(t),
\end{align*}
so $\pi_\alpha(a)\xi$ is in the domain of $D$. On the other hand, $(\pi_\alpha(a)D\xi)(t)=(2\pi i)^{-1}\alpha_{-t}(a)\partial\xi(t)/\partial t_k$. Thus 
$$
([D_k,\pi_\alpha(a)]\xi)(t)=\frac{1}{2\pi i}\alpha_{-t}(\delta_k(a))\xi(t),
$$
and the formula for the commutator with $\slashed{D}=\sum_k\gamma^kD_k$ follows.  
\end{proof}

The following lemma was proven in \cite{CP1} and is a very important result for the interplay between spectral triples and $KK$-theory.
\begin{Lemma}[{\cite[Lemma 2.3]{CP1}}]\label{BJwereright} 
Let $\slashed{D}$ be an unbounded selfadjoint operator on a Hilbert space $\Hi$, and let $\Dom(\slashed{D})$ be the domain of $\slashed{D}$. Suppose that $T\in\Bi(\Hi)$ maps $\Dom(\slashed{D})$ into itself. Then, for each $\lambda\in[0,\infty)$,
$$
[T,(\bone+\lambda+\slashed{D}^2)^{-1}]=\slashed{D}(\bone+\lambda+\slashed{D}^2)^{-1}[\slashed{D},T](\bone+\slashed{D}^2)^{-1}+(\bone+\lambda+\slashed{D}^2)^{-1}[\slashed{D},T]\slashed{D}(\bone+\lambda+\slashed{D}^2)^{-1}
$$
is an equality in $\Bi(\Hi)$.
\end{Lemma}
In the next lemma we write $D:=(D_1,\dots,D_n)$ and $|D|:=\sqrt{D_1^2+\cdots+D_n^2}$.
\begin{Lemma}\label{Riesmultiplier} Let $\slashed{D}$ be the Dirac operator \eqref{Diracoperator} associated with the $C^*$-dynamical system $(A,\R^n,\alpha)$ and define
$$
\slashed{F}:=\slashed{D}(\bone+\slashed{D}^2)^{-1/2}.
$$
Then for each $a\in A$, the commutator $[\slashed{F},\pi_\alpha(a)]$ belongs to $\Mn_N(B)$.
\end{Lemma}
\begin{proof} For $a\in A$, we know e.g. from Proposition \ref{easycommutator} that $[\slashed{D},\pi_\alpha(a)]$ is a bounded operator on $\Hi$, in fact a multiplier of $\Mn_N(B)$. This fact allows us to write
$$
[\slashed{F},\pi_\alpha(a)]=[\slashed{D},\pi_\alpha(a)](\bone+\slashed{D}^2)^{-1/2}+\slashed{D}[(\bone+\slashed{D}^2)^{-1/2},\pi_\alpha(a)].
$$
For every $\varphi\in C_0(\R)$ and $x\in A^\sim$, the operator $\pi_\alpha(x)\varphi(\slashed{D})$ is in $\Mn_N(B)$, as follows from the definition of the crossed product and the expression of $\slashed{D}$ in terms of the generators $D_1,\dots,D_n$. Therefore, the term $[\slashed{D},\pi_\alpha(a)](\bone+\slashed{D}^2)^{-1/2}$ is in $\Mn_N(B)$. It remains to show that we also have
$$
\slashed{D}[(\bone+\slashed{D}^2)^{-1/2},\pi_\alpha(a)]\in\Mn_N(B).
$$
For that, we use \cite[Remark A.3]{CP1} to write
$$
(\bone+\slashed{D}^2)^{-1/2}=\frac{1}{\pi}\int^\infty_0(\bone+\lambda+\slashed{D}^2)^{-1}\lambda^{-1/2}\, d\lambda,
$$
where the right-hand side converges in the norm on $\Bi(\Hi)$. By Lemma \ref{BJwereright} we then have
\begin{align*}
&\slashed{D}[(\bone+\slashed{D}^2)^{-1/2},\pi_\alpha(a)]
\\&=\frac{1}{\pi}\slashed{D}\int^\infty_0(\bone+\lambda+\slashed{D}^2)^{-1}\big([\slashed{D},\pi_\alpha(a)]\slashed{D}+\slashed{D}[\slashed{D},\pi_\alpha(a)]\big)(\bone+\lambda+\slashed{D}^2)^{-1}\lambda^{-1/2}\, d\lambda.
\end{align*}
Again we have convergence in the operator norm, so we can actually move the prefactor $\slashed{D}$ under the integral sign to obtain
\begin{align*}
\slashed{D}[(\bone+\slashed{D}^2)^{-1/2},\pi_\alpha(a)]
&=\frac{1}{\pi}\int^\infty_0\slashed{D}(\bone+\lambda+\slashed{D}^2)^{-1}[\slashed{D},\pi_\alpha(a)]\slashed{D}(\bone+\lambda+\slashed{D}^2)^{-1}\lambda^{-1/2}\, d\lambda
\\&+\frac{1}{\pi}\int^\infty_0\slashed{D}^2(\bone+\lambda+\slashed{D}^2)^{-1}[\slashed{D},\pi_\alpha(a)](\bone+\lambda+\slashed{D}^2)^{-1}\lambda^{-1/2}\, d\lambda,
\end{align*}
The whole integrand is in $\Mn_N(B)$ because, for instance, the operator $\slashed{D}^2(\bone+\lambda+\slashed{D}^2)^{-1}$ is bounded with norm $\leq 1$ and a multiplier of $\Mn_N(B)$. Moreover, the estimates \cite[Remark 5]{CGPRS}
$$
\|(\bone+\lambda+\slashed{D}^2)^{-1}\|\leq \frac{1}{1+\lambda},\qquad \|\slashed{D}(\bone+\lambda+\slashed{D}^2)^{-1}\|\leq \frac{1}{2\sqrt{1+\lambda}},
$$
which follow from functional calculus, show that the integral is norm convergent. That completes the proof.
\end{proof}

\begin{Remark}\label{evenremark}
As mentioned, for even $n$ we can always find a grading operator $\Gamma$ on $\Hi$ such that $\Gamma\pi_\alpha(a)=\pi_\alpha(a)\Gamma$ for all $a\in A$ and $\Gamma\slashed{D}=-\slashed{D}\Gamma$. In the example $n=2$ we can take $\Gamma$ to be $\diag(1,-1)$. We write
$$
\Hi=\Hi_+\oplus\Hi_-
$$
for even $n$, with $\Hi_{\pm}$ the $\pm 1$-eigenspace of the grading operator $\Gamma$. Under the decomposition $\Hi=\Hi_+\oplus\Hi_-$, the algebra $\Mn_N(B)$ splits as $\Mn_N(B)=\Mn_N(B)_+\oplus \Mn_N(B)_-$, where $\Mn_N(B)_+$ is the part of $\Mn_N(B)$ commuting with the grading operator $\Gamma=\diag(\bone,-\bone)$ and $\Mn_N(B)_-$ is the part anti-commuting with $\Gamma$. In turn, this induces an even grading $\Mn_N(\boldsymbol B)=\Mn_N(\boldsymbol B)_+\oplus\Mn_N(\boldsymbol B)_-$ of the Hilbert $B$-module $\Mn_N(\boldsymbol B)$. 
\end{Remark}

We let $\pi_B:A\to\Mn_N(\C)\otimes\Mi(B)$ be the representation of $A$ which takes $a\in A$ to the operator of left multiplication by the multiplier $1_N\otimes\pi_\alpha(a)$ of $\Mn_N(\C)\otimes B$. The operators $\pi_B(a)$ are even for the grading of the Hilbert $B$-module $\Mn_N(B)$, whereas $\gamma^1,\dots,\gamma^n$ are odd. These observations lead to the following result.
\begin{prop}\label{crucialcorrollary} 
The triple $(\pi_B,\Mn_N(B),\slashed{F})$ is a Kasparov $A$-$B$-module and defines a class 
$$
\mathbf{t}_\alpha\in KK^\bullet(A,B),
$$
where $\bullet\in\{0,1\}=\{\rm even,\rm odd\}$ is the parity of $n$.

Let $\slashed{\boldsymbol{R}}:=\slashed{\boldsymbol{D}}|\slashed{\boldsymbol{D}}|^{-1}$ denote the phase of the massive Dirac operator $\slashed{\boldsymbol{D}}$. Then $(\boldsymbol\pi_B, \Mn_N(\boldsymbol B),\slashed{\boldsymbol{R}})$ is an even Kasparov $A$-$(\C_n\otimes B)$-module and defines the same class $\mathbf{t}_\alpha$ in $KK^\bullet(A,B)$.
\end{prop}
\begin{proof} 
We have seen in Lemma \ref{Riesmultiplier} that $[\slashed{F},\pi_\alpha(a)]$ is in $\Mn_N(B)$ for all $a\in A$. So for the first statement it remains only to show that $\pi_B(A)(\slashed{F}^2-\bone)$ is contained in $\Mn_N(B)$. For that, let $a\in A$ and write
$$
\pi_B(a)(\slashed{F}^2-\bone)=\pi_B(a)(\bone+\slashed{D}^2)^{-1}=\pi_B(a)\varphi(\slashed{D})
$$
where $\varphi:\R\to \C$ vanishes at infinity. Thus we have $\pi_B(a)(\slashed{F}^2-\bone)\in\Mn_N(B)$.

The same proof as that of Lemma \ref{Riesmultiplier} shows that $[\slashed{\boldsymbol R},\boldsymbol\pi_B(a)]$ belongs to $\Mn_N(\boldsymbol B)$ for all $a\in A$. That $\mathbf{t}_\alpha$ is also represented by $(\boldsymbol\pi_B,\Mn_N(\boldsymbol B),\slashed{\boldsymbol{R}})$ follows from \cite[Lemma 2.10]{CGRS1}.
\end{proof}


\subsection{The Toeplitz extension}\label{Toep}
We are now in position to deduce the Toeplitz extension.

\begin{proof}[Proof of Proposition \ref{exactsequence}]
Regard $\Mn_N(\boldsymbol B)$ as a subalgebra of $B\otimes\Ki$, where $\Ki$ is the $C^*$-algebra of compact operators on some infinite-dimensional Hilbert space. From Proposition \ref{crucialcorrollary} we know that $\slashed{\boldsymbol R}$ and $\slashed{\boldsymbol P}$ are multipliers of $B\otimes\Ki$. So we have a projection $\slashed{\boldsymbol P}\in\Mi(B\otimes\Ki)$ with $[\slashed{\boldsymbol P},\boldsymbol{\pi}_\alpha(A)]\subset B\otimes\Ki$. We know that this characterizes an invertible extension. The Busby invariant of this extension (cf. \cite[\S15]{Bla}, \cite{Busby1}, \cite[\S3.2]{Ols}) is given by $\boldsymbol\gamma_\alpha(a):=q(\slashed{\boldsymbol P}\boldsymbol{\pi}_\alpha(a)\slashed{\boldsymbol P})$, where $q:\Mi(B\otimes\Ki)\to\Qi(B\otimes\Ki)$ is the Calkin map. The proof is complete by noticing that the pullback $C^*$-algebra associated to $\boldsymbol\gamma_\alpha$ (cf. \cite[Prop. 3.2.11]{Ols}),
$$
\boldsymbol\Ti\cong\{(T,a)\in\Mi(B\otimes\Ki)\oplus A|\ q(T)=\boldsymbol\gamma_\alpha(a)\},
$$
is indeed the Toeplitz $C^*$-algebra. 
\end{proof}
We refer to the exact sequence in Proposition \ref{exactsequence} as the \textbf{Toeplitz extension} of $(A,\R^n,\alpha)$. It determines an element of $\Ext(A,B)^{-1}$, the group of semisplit extensions of $A$ by $B\otimes\Ki$. Recall \cite[Lemma 6.2]{Kasp1}, which says that
$$
KK^1(A,B)\cong \Ext(A,B)^{-1},
$$ 
where $\Ext(A,B)^{-1}$ is the group of invertible elements in the semigroup $\Ext(A,B)$ of extensions of $A$ by $B$. 
As we saw in the proof of Proposition \ref{exactsequence}, the Busby invariant of the Toeplitz extension of $(A,\R^n,\alpha)$ is given by 
$$
\boldsymbol\gamma_\alpha(a):=q(\slashed{\boldsymbol P}\boldsymbol{\pi}_\alpha(a)\slashed{\boldsymbol P}).
$$
For odd $n$, it follows that the class of the Toeplitz extension identitifes with the element of $KK^1(A,B)$ denoted by $\mathbf{t}_\alpha$ in Proposition \ref{crucialcorrollary}. We shall see in the next section that $\mathbf{t}_\alpha$ is in fact the Thom element for $(A,\R^n,\alpha)$, both for even and odd $n$.
\begin{Remark}
By Proposition \ref{crucialcorrollary}, the ideal of $\boldsymbol\Ti$ generated by the elements
$$
\slashed{\boldsymbol P}\boldsymbol{\pi}_\alpha(a)\boldsymbol{\pi}_\alpha(b)\slashed{\boldsymbol P}-\slashed{\boldsymbol P}\boldsymbol{\pi}_\alpha(a)\slashed{\boldsymbol P}\boldsymbol{\pi}_\alpha(b)\slashed{\boldsymbol P},\qquad a,b\in A
$$
coincides with $\Mn_N(\boldsymbol B)$, which is another way of seeing that $\Mn_N(\boldsymbol B)$ is an ideal in $\boldsymbol\Ti$. 
\end{Remark}


\begin{Lemma}\label{Fredlemma} The operator $T_a:=\slashed{\boldsymbol P}\boldsymbol{\pi}_\alpha(a)\slashed{\boldsymbol P}$ is Fredholm as an operator on $\Mn_N(\boldsymbol B)$ iff $a$ is invertible in $A^\sim$. 
\end{Lemma}
\begin{proof} From Proposition \ref{exactsequence} it follows that if $T_a$ is invertible modulo $\Mn_N(\boldsymbol B)$ then $a$ is invertible. Conversely, if $u\in\Ai^\sim$ is invertible then, since $[\slashed{\boldsymbol P},\boldsymbol{\pi}_\alpha(u)]\in\Mn_N(\boldsymbol B)$ by Lemma \ref{Riesmultiplier}, we get
$$
(\slashed{\boldsymbol P}\boldsymbol{\pi}_\alpha(u)\slashed{\boldsymbol P})(\slashed{\boldsymbol P}\boldsymbol{\pi}_\alpha(u^{-1})\slashed{\boldsymbol P})\equiv\slashed{\boldsymbol P}\text{ mod }\Mn_N(\boldsymbol B),
$$
and similarly for $u\leftrightarrow u^{-1}$. Now $\slashed{\boldsymbol P}$ is the identity in $\slashed{\boldsymbol P}\Mn_N(\boldsymbol\Ni)\slashed{\boldsymbol P}$, where $\boldsymbol\Ni:=\boldsymbol B''$. 
\end{proof}
We shall later define a $K_0(B)$-valued index for $\Mn_N(\boldsymbol B)$-relative Fredholm operators in $\slashed{\boldsymbol P}\pi_\alpha(A^\sim)\slashed{\boldsymbol P}$.

\subsection{The Thom class}\label{Extsection}
In this section we show that $(\pi_\alpha,\Mn_N(B),\slashed{F})$ is a representative of the Thom class for the $C^*$-dynamical system $(A,\R^n,\alpha)$.  

First we need to recall some fundamental facts about crossed products. Let $B:=A\rtimes_\alpha\R^n$ be the crossed product. There is an action $\hat{\alpha}:\hat{\R}^n\to\Aut(B)$ on $B$ of the dual group $\hat{\R}^n\cong\R^n$ of $\R^n$, called the \textbf{dual action} \cite[Def. X.2.4]{Ta2}, characterized by ($s\in\hat{\R}^n$)
$$
\hat{\alpha}_s(\pi_\alpha(a)):=\pi_\alpha(a),\qquad\forall\, a\in A,
$$
$$
\hat{\alpha}_s(\lambda_t):=e^{-2\pi is\cdot t}\lambda_t,\qquad\forall\, t\in\R^n.
$$
The fixed-point subalgebra of $B$ under the action $\hat{\alpha}$ is just $\pi_\alpha(A)$. A fundamental fact is that iterating the crossed-product construction using the dual action gives back $A$ (up to stable isomorphism).
\begin{thm}[Takesaki-Takai duality {\cite[Thm. X.2.3]{Ta2}, \cite{Takai1}}]\label{thmTakai}
Let $(A,\R^n,\alpha)$ be a $C^*$-dynamical system. Then the crossed product of $A\rtimes_\alpha\R^n$ with $\hat{\R}^n$ by the dual action $\hat{\alpha}$ is stably isomorphic to the original algebra $A$:
$$
(A\rtimes_\alpha\R^n)\rtimes_{\hat{\alpha}}\hat{\R}^n\cong A\otimes\Ki.
$$
On the level of von Neumann algebras $\Mi=A''$ and $\Ni=(A\rtimes_\alpha\R^n)''$, the duality reads
$$
\Ni\rtimes_{\hat{\alpha}}\hat{\R}^n\cong \Mi\otimes\Bi(L^2(\R^n)).
$$
The isomorphism can be chosen so that the double dual action $\hat{\hat{\alpha}}$ is intertwined with the action $\alpha\otimes\Ad(\lambda)$ on $A\otimes\Ki$, where $\Ad(\lambda_t)(T):=\lambda_{-t}T\lambda_t$ for $T\in\Ki(L^2(\R^n))$.
\end{thm}
In view of Theorem \ref{thmTakai}, we refer to $(B,\hat{\R}^n,\hat{\alpha})$ as the \textbf{dual dynamical system} of $(A,\R^n,\alpha)$.

We want the construction of the Kasparov $A$-$B$-module $(\pi_\alpha,\Mn_N(B),\slashed{F})$ from the data $(A,\R^n,\alpha)$ to  be ``compatible" with Takesaki-Takai duality. Otherwise the element $\mathbf{t}_{\hat{\alpha}}$ associated with the dual dynamical system $(B,\hat{\R}^n,\hat{\alpha})$ will not be an inverse for $\mathbf{t}_\alpha$. To have this compatibility we need to be a little bit more cunning and distinguish between $\R^n$-actions and actions by the dual group $\hat{\R}^n$. Thus, we make the following convention.
\begin{dfn}\label{defofdualthom}
As before, let $\C_n$ be the complex Clifford algebra associated with the vector space $\R^n$ equipped with the standard Euclidean inner product $\bra\cdot|\cdot\ket$. We let $\C_{-n}$ be complex Clifford algebra associated with $(\R^n,-\bra\cdot|\cdot\ket)$, i.e. with $\R^n$ equipped with the negative inner product. Let $\hat{\gamma}^1,\dots,\hat{\gamma}^n$ be the skew-Hermitian generators of the irreducible representation of $\C_{-n}$ on $\C^N$. 

Let $(B,\hat{\R}^n,\hat{\alpha})$ be a $C^*$-dynamical system, i.e. a $C^*$-algebra $B$ equipped with a strongly continuous action $\hat{\alpha}:\hat{\R}^n\to\Aut(B)$. Let $X_1,\dots,X_n$ be the generators of the unitary group implementing $\hat{\alpha}$ in the representation $\pi_{\hat{\alpha}}$. We define the Kasparov $B$-$(B\rtimes_{\hat{\alpha}}\hat{\R}^n)$-module 
$$
(\pi_{\hat{\alpha}},\Mn_N(B\rtimes_{\hat{\alpha}}\hat{\R}^n),\slashed{X}(\bone+\slashed{X}^2)^{-1/2})
$$
just as the module $(\pi_\alpha,\Mn_N(A\rtimes_\alpha\R^n),\slashed{D}(\bone+\slashed{D}^2)^{-1/2})$ was defined in the case of $\R^n$-actions (\S\ref{Diracsection}), but now with
\begin{equation}\label{dualactionDirac}
\slashed{X}:=\sqrt{-1}\sum^n_{k=1}\hat{\gamma}^kX_k
\end{equation}
playing the role of $\slashed{D}$. We let $\mathbf{t}_{\hat{\alpha}}$ denote the class in $KK^\bullet(B,B\rtimes_{\hat{\alpha}}\hat{\R}^n)$ defined by the module $(\pi_{\hat{\alpha}},\Mn_N(B\rtimes_{\hat{\alpha}}\hat{\R}^n),\slashed{X}(\bone+\slashed{X}^2)^{-1/2}$. Here $\bullet\in\{0,1\}=\{\rm even,\rm odd\}$ is the parity of $n$. 
\end{dfn}
The operator $\hat{\gamma}^k$ on $\C^N$ anticommutes with $\gamma^j$ for each $j,k=1,\dots,n$. We shall see in the proof of the following why that is important. 
\begin{prop}
The class $\mathbf{t}_\alpha$ of the Kasparov $A$-$B$ module $(\pi_\alpha,\Mn_N(B),\slashed{F})$ is the Thom class of $(A,\R^n,\alpha)$. In particular, $\mathbf{t}_\alpha\in KK^\bullet(A,B)$ is a $KK$-equivalence of degree shift $n$, with inverse $\hat{\mathbf{t}}_\alpha:=\mathbf{t}_{\hat{\alpha}}\in KK^\bullet(B,B\rtimes_{\hat{\alpha}}\hat{\R}^n)= KK^\bullet(B,A\otimes\Ki)= KK^\bullet(B,A)$. 
\end{prop}
\begin{proof} We have to show that the class $\mathbf{t}_\alpha$ satisfies the axioms of the Thom class similar to those stated in \cite{FS}, \cite[\S2]{Rose4} for $1$-parameter actions. Most of the proof is very similar to the case $n=1$ but worth spelling out in detail (for $n=1$ there is no need to distinguish between $\R$ and $\hat{\R}$-actions).

\begin{proof}[Normalization]
Let $A=\C$, so that $\alpha$ is the trivial $\R^n$-action and $B=C_0(\R^n)$. We need to show that
\begin{equation}\label{normequli}
\mathbf{t}_\alpha\otimes_\C\hat{\mathbf{t}}_\alpha=1_\C,\qquad \hat{\mathbf{t}}_\alpha\otimes_\C\mathbf{t}_\alpha=1_B.
\end{equation}
But in this case, $\mathbf{t}_\alpha$ and $\hat{\mathbf{t}}_\alpha$ are the ``Dirac" and ``Dirac-dual" elements for $\R^n$ \cite[Def. 4.2]{Kasp3} and the equalities \eqref{normequli} are equivelent to Bott periodicity in $KK$ \cite[Thm. 5.7]{Kasp1}. So the result is well known. Let us just sketch the idea, so that we see the motivation for Definition \ref{defofdualthom}.

Both $B$ and the iterated crossed product $B\rtimes_{\hat{\alpha}}\R^n\cong\Ki$ act on $L^2(\R^n)$. Let $X_1,\dots,X_n$ be the generators of the unitary group implementing the dual action $\hat{\alpha}$ in $L^2(\R^n)$, which is the action of $\R^n$ by translations on $B$. 


As defined in the last section, the element $\mathbf{t}_\alpha$ is represented by $(\pi_\alpha,\Hi,\slashed{D})$, where $\slashed{D}=\sum_k\gamma^kD_k$ for unbounded selfadjoint operators $D_1,\dots,D_n$ on $L^2(\R^n)$ such that $[X_j,D_k]=\sqrt{-1}$. Consider the operator
$$
\slashed{K}:=\slashed{D}+\slashed{X}.
$$
Definition \ref{defofdualthom} ensures that $\slashed{K}^2$ is (minus a bounded normal operator) the $n$-dimensional harmonic oscillator, which has discrete spectrum. In particular, $(\bone+\slashed{K})^{-1}$ is compact. Thus $(\pi_\C,\Hi,\slashed{K}(\bone+\slashed{K}^2)^{-1/2})$ is a Kasparov $\C$-$\C$-module, where $\pi_\C(\lambda):=\lambda\bone$. In fact, $(\pi_\C,\Hi,\slashed{K}(\bone+\slashed{K}^2)^{-1/2})$ represents the Kasparov product $\mathbf{t}_\alpha\otimes_B\hat{\mathbf{t}}_\alpha$ \cite[Thm. 5.7]{Kasp1}. Moreover, $\slashed{K}$ is surjective and its kernel is the $1$-dimensional subspace spanned by the vector $\xi_0(t):=e^{-|t|^2}$. So $\slashed{K}$ is Fredholm, with Fredholm index $1$, and it represents the generator $[1_\C]\in KK^\bullet(\C,\C)$. So $\mathbf{t}_\alpha\otimes_B\hat{\mathbf{t}}_\alpha=1_\C$. The second equality in \eqref{normequli} follows from a version of Atiyah's rotation trick \cite{Atiy3} or, alternatively, Takesaki-Takai duality (cf. below in the last paragraph of this proof).
\end{proof}
\begin{proof}[Naturality]
Let $\rho:(A,\alpha)\to (A',\alpha')$ be an equivariant homomorphism of $C^*$-dynamical systems. Then $\rho$ induces a $*$-homomorphism $\hat{\rho}:B\to B'$ of the crossed products $B:=A\rtimes_\alpha\R^n$ and $B':= A'\rtimes_{\alpha'}\R^n$. For $f\in L^1(\R^n,A)$, one sets
$$
\hat{\rho}(\hat{\pi}_\alpha(f)):=\int_{\R^n}(\pi_\alpha\circ\rho)\big(f(t)\big)e^{-2\pi i t\cdot D}\, dt,
$$
where we recall that $\hat{\pi}_\alpha(f):=\int_{\R^n}\pi_\alpha\big(f(t)\big)e^{-2\pi i t\cdot D}\, dt$. In particular, we have
$$
\hat{\rho}\circ\pi_\alpha=\pi_{\alpha'}\circ\rho
$$
as maps from $A$ into $\Mi(A'\rtimes_{\alpha'}\R^n)$. The homomorphism $\rho$ defines a class $[\rho]\in KK^0(A,A')$ \cite[\S 17.1.2(a)]{Bla}. Let $\rho^*$ and $\rho_*$ be the maps on $KK$ given by left and right Kasparov product with the class $[\rho]$, respectively. Similarly, the ``dual" homomorphism $\hat{\rho}$ gives rise to a $KK$-class $[\hat{\rho}]\in KK^0(B,B')$ and we define $\hat{\rho}^*:=[\hat{\rho}]\otimes_{B'}\cdot$ and $\hat{\rho}_*:=\cdot\otimes_B[\hat{\rho}]$ as the operations of Kasparov product with the class $[\hat{\rho}]$.

In this notation, the naturality the Thom classes have to satisfy is the equality
$$
\hat{\rho}_*(\mathbf{t}_\alpha)=\rho^*(\mathbf{t}_{\alpha'})
$$
for all equivariant maps $\rho:(A,\alpha)\to (A',\alpha')$. Thus, we need to show that
$$
\hat{\rho}_*[\pi_\alpha,\Mn_N(B),\slashed{F}]:=[\pi_\alpha\otimes\id,\Mn_N(B)\otimes_{\hat{\rho}}\Mn_N(B'),\slashed{F}\otimes\bone]
$$
coincides with 
$$
\rho^*[\pi_{\alpha'},\Mn_N(B'),\slashed{F}']:=[\pi_{\alpha'}\circ\rho,\Mn_N(B'),\slashed{F}'].
$$
By definition of the balanced tensor product, $\Mn_N(B)\otimes_{\hat{\rho}}\Mn_N(B')=\overline{\hat{\rho}(\Mn_N(B))}$ is the closed right ideal in $\Mn_N(B')$ generated by $\hat{\rho}(\Mn_N(B))$ and $\pi\otimes\id$ becomes the representation $\hat{\rho}\circ\pi_\alpha$. Now $\hat{\rho}\circ\pi_\alpha=\pi_{\alpha'}\circ\rho$ gives the result. 
\end{proof}

\begin{proof}[Compatibility with external products]
We need to show that, for all $\mathbf{x}\in KK^0(A',A)$ and $\mathbf{y}\in KK^0(C',C)$,
$$
\mathbf{y}\boxtimes(\mathbf{x}\otimes_A\mathbf{t}_\alpha)=(\mathbf{y}\boxtimes\mathbf{x})\otimes_{C\otimes A}\mathbf{t}_{\id_C\otimes\alpha}.
$$
This property is clearly satisfied by $\mathbf{t}_\alpha=[\pi_\alpha,\Mn_N(B),\slashed{F}]$. For instance, take $C=C'$ and $\mathbf{y}=1_C$. Then $1_C\boxtimes\mathbf{t}_{\alpha}=\mathbf{t}_{\id_C\otimes\alpha}$. The general case follows by definition of $\boxtimes$.
\end{proof}
Thus, we have shown that $\mathbf{t}_\alpha$ satisfies the higher-dimensional analogue of the Fack-Skandalis axioms for the Thom element. The next task is to show that these axioms implies that $\mathbf{t}_\alpha$ is a $KK$-equivalence. The proof \cite[\S 19.3]{Bla}, \cite[Thm. 2.3]{Rose4}, \cite{FS} that $\mathbf{t}_\alpha$ is a $KK$-equivalence carries over completely. For completeness we reproduce the details. 

For each $\lambda\in[0,1]$ we have the rescaled $\R^n$-action
$$
\R\ni t\to \alpha^\lambda_t:=\alpha_{\lambda t},
$$
where $\lambda(t_1,\dots,t_n):=(\lambda t_1,\dots,\lambda t_n)$. We note that for $\lambda=1$ we have the original action $\alpha$ while $\alpha^0$ is the trivial action. Consider the $C^*$-algebra $A'=C([0,1],A)$ and the $\R^n$-action 
$$
(\alpha'_t(f))(\lambda):=\alpha^\lambda_t(f(\lambda))
$$
on $A'=C([0,1])\otimes A$. We use the shorthand notation $B':=A'\rtimes_{\alpha'}\R^n$ and $\mathbf{t}'_\alpha:=\mathbf{t}_{\alpha'}$. The evaluation $\rho_\lambda:A'\to A$, given by $\rho_\lambda(f):=f(\lambda)$, is equivariant: $\alpha^\lambda\circ\rho_\lambda=\alpha'$.  
By naturality of the Thom elements, we therefore have
$$
(\rho_\lambda)_*(\hat{\mathbf{t}}_{\alpha}')=\hat{\rho}_\lambda^*(\hat{\mathbf{t}}_{\alpha^\lambda}),\qquad
(\hat{\rho}_\lambda)_*(\mathbf{t}_{\alpha}')=\rho_\lambda^*(\mathbf{t}_{\alpha^\lambda}).
$$
As above, the maps $\rho^*_\lambda$ and $(\rho_\lambda)_*$ are defined as the operations of left and right Kasparov product with a class $[\rho_\lambda]\in KK^0(A',A)$. 
By associativity of the Kasparov product, we have
\begin{align*}
(\rho_\lambda)_*(\mathbf{t}_{\alpha}'\otimes_{B'}\hat{\mathbf{t}}_{\alpha}')&=\mathbf{t}_{\alpha}'\otimes_{B'}(\rho_\lambda)_*(\hat{\mathbf{t}}_{\alpha}')
\\&=\mathbf{t}_{\alpha}'\otimes_{B'}\hat{\rho}_\lambda^*(\hat{\mathbf{t}}_{\alpha^\lambda})
\\&=(\hat{\rho}_\lambda)_*(\mathbf{t}_{\alpha}')\otimes_{B'}\hat{\mathbf{t}}_{\alpha^\lambda}
\\&=\rho_\lambda^*(\mathbf{t}_{\alpha^\lambda})\otimes_{B'}\hat{\mathbf{t}}_{\alpha^\lambda}
\\&=\rho_\lambda^*(\mathbf{t}_{\alpha^\lambda}\otimes_{B'}\hat{\mathbf{t}}_{\alpha^\lambda})
\\&=[\rho_\lambda]\otimes_A(\mathbf{t}_{\alpha^\lambda}\otimes_{B'}\hat{\mathbf{t}}_{\alpha^\lambda}).
\end{align*}
The family $(\rho_\lambda)_{\lambda\in[0,1]}$ is a continuous homotopy so each $\rho_\lambda$ induces the same map $(\rho_\lambda)_*$ on $KK$.  So for any $\lambda\in[0,1]$ we have
$$
(\mathbf{t}_{\alpha}'\otimes_{B'}\hat{\mathbf{t}}_{\alpha}')\otimes_{A'}[\rho_0]
=[\rho_0]\otimes_A(\mathbf{t}_{\alpha^\lambda}\otimes_{B'}\hat{\mathbf{t}}_{\alpha^\lambda}).
$$
Now $\rho_0(f)=f(0)$ is the evaluation at the endpoint, and the map $\iota(a):=a\otimes\bone$ is a homotopy inverse to $\rho_0$. So we have an inverse $\iota^*=[\iota]\otimes_{A'}$ to $\rho_0^*=[\rho_0]\otimes_A$ and we can make the rearrangement
$$
[\iota]\otimes_{A'}(\mathbf{t}_{\alpha}'\otimes_{B'}\hat{\mathbf{t}}_{\alpha}')\otimes_{A'}[\rho_0]
=\mathbf{t}_{\alpha^\lambda}\otimes_{B}\hat{\mathbf{t}}_{\alpha^\lambda}.
$$
The left-hand side is independent of $\lambda\in[0,1]$, so the right-hand side must be independent of $\lambda\in[0,1]$ as well. But for $\lambda=0$ we know from the normalization and naturality axioms that $\mathbf{t}_{\alpha^0}\otimes_{B}\hat{\mathbf{t}}_{\alpha^0}=1_A$. At $\lambda=1$ we obtain the desired result $\mathbf{t}_{\alpha}\otimes_{B}\hat{\mathbf{t}}_{\alpha}=1_A$.

Using Takesaki-Takai duality we have $\hat{\hat{\mathbf{t}}}_\alpha=\mathbf{t}_\alpha\otimes_\C 1_\Ki$. Since $\hat{\rho}:=\hat{\rho}_\lambda$ leaves $C^*(\R^n)\subset B'$ untouched and intertwines $\pi_{\alpha'}$ with $\pi_\alpha$, we have again an equivariant map
$$
\hat{\rho}:(B',\hat{\alpha'})\to (B,\hat{\alpha})
$$
of the dual dynamical systems. Therefore, we can iterate the process and obtain a map $\hat{\hat{\rho}}:B\rtimes_{\hat{\alpha}}\hat{\R}^n\to B'\rtimes_{\hat{\alpha'}}\hat{\R}^n$ between the iterated crossed products. Under the isomorphism $B\rtimes_{\hat{\alpha}}\hat{\R}^n\cong A\otimes\Ki$ one checks that $\hat{\hat{\rho}}$ becomes
$$
\hat{\hat{\rho}}=\rho\otimes\id.
$$
So by replacing $\alpha$ with $\hat{\alpha}$ we obtain $\hat{\mathbf{t}}_{\alpha}\otimes_A\mathbf{t}_{\alpha}=1_B$. That finishes the proof.
\end{proof}

\subsection{Kasparov products with the Thom class} 
We shall show that the Kasparov product $[x]\otimes_A[\boldsymbol\pi_B,\boldsymbol B,\slashed{\boldsymbol{R}}]$ with $K$-theory classes $[x]\in K_\bullet(A)$ is equal to the $K_0(B)$-valued index of a Fredholm operator on the Hilbert $B$-module $\boldsymbol{B}_B$. 

\begin{Notation}
In the following statements we want to allow for matrices over $A$. In order to make the formulas readable, we shall make the following convention. For $x\in\Mn_r(A)=\Mn_r(\C)\otimes A$, we write
$$
\boldsymbol\pi_B(x):=(\id\otimes\boldsymbol\pi_B)(x),\qquad \slashed{\boldsymbol{R}}\boldsymbol\pi_B(x):=(1_r\otimes\slashed{\boldsymbol{R}})\boldsymbol\pi_B(x)
$$ 
as operators on $\C^r\otimes\boldsymbol{B}=\boldsymbol{B}^{\oplus r}$. There should be no confusion since without this convention, the expression $\boldsymbol\pi_B(x)$ etc. does not make sense unless $r=1$.
\end{Notation}

\begin{thm}\label{oddKaspprod} 
Suppose that $n$ is odd. Let $u\in\Un_r(A^\sim)$ be a unitary over $A^\sim$ and denote by $[u]\in K_1(A)$ the homotopy class of $u$. Then we have the equality
$$
[u]\otimes_A[\boldsymbol\pi_B,\boldsymbol B,\slashed{\boldsymbol{R}}]=\Index(\slashed{\boldsymbol{P}}\boldsymbol\pi_B(u)\slashed{\boldsymbol{P}})
$$
in $K_0(B)$.
\end{thm}
\begin{proof} For ease of notation, assume $r=1$. 

To a unitary $u\in A^\sim$ there corresponds a homomorphism $\rho_u:C_0(\R)^{\sim}\to\Mi(A)$ which takes $z-\bone$ to $u-\bone$ under the identification of $K_1(A)$ with $KK^0(C_0(\R),A)$. The Kasparov product of the class of $\rho_u$ with the element $[\boldsymbol\pi_B,\boldsymbol B,\slashed{\boldsymbol{R}}]$ is given by
$$
[\rho_u]\otimes_A[\boldsymbol\pi_B,\boldsymbol B,\slashed{\boldsymbol{R}}]=[\boldsymbol\pi_B\circ\rho_u,\boldsymbol B,\slashed{\boldsymbol{R}}],
$$
which is an element of $KK^1(C_0(\R),B)$. A homomorphism such as $\boldsymbol\pi_B\circ\rho_u:C_0(\R)^{\sim}\to\Mi(B\otimes\Ki)$ defines a unitary operator $U=\boldsymbol\pi_B\circ\rho_u(z)$ in $\Mi(B\otimes\Ki)$ (which in the present case is just $\boldsymbol\pi_B(u)$), and conversely a unitary in $\Mi(B\otimes\Ki)$ determines a homomorphism from $C_0(\R)^{\sim}$ to $\Mi(B\otimes\Ki)$. Homotopy equivalence of homomorphisms from $C_0(\R)^{\sim}$ to $\Mi(B\otimes\Ki)$ translates into homotopy equivalence of the corresponding unitaries in $\Mi(B\otimes\Ki)$. So the class $[\boldsymbol\pi_B\circ\rho_u,\boldsymbol B,\slashed{\boldsymbol{R}}]$ is represented by a unitary $U=\boldsymbol\pi_B(u)$ in $\Mi(B\otimes\Ki)$ which commutes with $\slashed{\boldsymbol{R}}$ modulo $B\otimes\Ki$. 

Therefore, if we set $\slashed{\boldsymbol{P}}:=(\bone+\slashed{\boldsymbol{R}})/2$ then the operator $\slashed{\boldsymbol{P}}\boldsymbol\pi_B(u)\slashed{\boldsymbol{P}}$ is a Fredholm operator, i.e. it is invertible modulo $B\otimes\Ki$. Under the isomorphism \cite[Prop. 12.2.1]{Bla}, \cite[Cor. 10.3]{Ols} 
$$
K_0(B)\cong K_1(\Qi(B\otimes\Ki)),\qquad K_1(B)\cong K_0(\Qi(B\otimes\Ki)),
$$
our class can be identified with the class $[q(\slashed{\boldsymbol{P}}\boldsymbol\pi_B(u)\slashed{\boldsymbol{P}})]$ in $K_1(\Qi(B\otimes\Ki))$, where $q:\Mi(B\otimes\Ki)\to\Qi(B\otimes\Ki)$ is the quotient map. 

What we have done so far is just to trace the fate of the representative $(\boldsymbol\pi_B\circ\rho_u,\boldsymbol B,\slashed{\boldsymbol{R}})$ under the isomorphism of $KK^1(C_0(\R),B)$ with $K_1(\Qi(B\otimes\Ki))$ \cite[Prop. 17.5.7]{Bla}. 

Equivalence classes of Fredholm operators on the Hilbert $B$-module $\ell^2(\N;B)$ correspond to elements in $K_0(B)$ via the index map (one of the connecting maps in the $6$-term exact sequence in $K$-theory). The image of $[\rho_u]\otimes_A[\boldsymbol\pi_B,\boldsymbol B,\slashed{\boldsymbol{R}}]$ under this map is the $K_0(B)$-valued index of the Fredholm operator $\slashed{\boldsymbol{P}}\boldsymbol\pi_B(u)\slashed{\boldsymbol{P}}$:
$$
\delta([q(\slashed{\boldsymbol{P}}\boldsymbol\pi_B(u)\slashed{\boldsymbol{P}})])
=[\Ker(\slashed{\boldsymbol{P}}\boldsymbol\pi_B(u)\slashed{\boldsymbol{P}})]
-[\Ker(\slashed{\boldsymbol{P}}\boldsymbol\pi_B(u^*)\slashed{\boldsymbol{P}})].
$$
\end{proof}
As mentioned in Remark \ref{evenremark}, if the spectral triple $(\Ai,\Hi,\slashed{D})$ is even then so is the doubled triple $(\Ai,\boldsymbol\Hi,\slashed{\boldsymbol{D}})$. The phase $\slashed{\boldsymbol{R}}=\slashed{\boldsymbol{D}}|\slashed{\boldsymbol{D}}|^{-1}$ decomposes in $\boldsymbol\Hi=\boldsymbol\Hi_+\oplus\boldsymbol\Hi_-$ as
$$
\slashed{\boldsymbol{R}}=\begin{pmatrix}0&\slashed{\boldsymbol{R}}_-\\\slashed{\boldsymbol{R}}_+&0\end{pmatrix}
$$
with $\slashed{\boldsymbol{R}}_-=(\slashed{\boldsymbol{R}}_+)^*$. Under the decomposition $\Hi=\Hi_+\oplus\Hi_-$, the algebra $B$ splits as $B=B_+\oplus B_-$, and this induces an even grading $\boldsymbol B=\boldsymbol B_+\oplus\boldsymbol B_-$ of the Hilbert $B$-module $\boldsymbol B$. Here $\boldsymbol B_+$ is the part of $\boldsymbol B$ commuting with the grading operator $\Gamma=\diag(\bone,-\bone)$ and $\boldsymbol B_-$ is the part anti-commuting with $\Gamma$ (cf. Remark \ref{evenremark}). 
\begin{thm}\label{evenKaspprod} 
Suppose that $(\Ai,\Hi,\slashed{D})$ is even. Let $e,f\in\Mn_\infty(\Ai^\sim)$ be projections over $\Ai^\sim$ with $[e]-[f]\in K_0(A)$. Then we have the equality
$$
([e]-[f])\otimes_A[\boldsymbol\pi_B,\boldsymbol B,\slashed{\boldsymbol{R}}]=\Index(\boldsymbol\pi_B(e)\slashed{\boldsymbol{R}}_+\boldsymbol\pi_B(e))
-\Index(\boldsymbol\pi_B(f)\slashed{\boldsymbol{R}}_+\boldsymbol\pi_B(f))
$$
in $K_0(B)$, where $\boldsymbol\pi_B(e)\slashed{\boldsymbol{R}}_+\boldsymbol\pi_B(e)$ is viewed as an operator from $\boldsymbol\pi_B(e)\boldsymbol{B}_+^{\oplus r}$ to $\boldsymbol\pi_B(e)\boldsymbol{B}_-^{\oplus r}$.
\end{thm}
\begin{proof} Again we consider matrices of size $r=1$ for simplicity. So let $e,f\in A^\sim$ be projections with $e-f\in A$.

An isomorphism $K_0(A)\cong KK^0(\C,A)$ (described e.g. in \cite[Prop. 2.4.26]{An3}) sends the class $[e]-[f]$ to (the $KK$-equivalence class of) a homomorphism $\rho$ from $\C$ to $\Mi(A\oplus A)$ sending $1\in\C$ to $e\oplus f$. The Kasparov product with the $KK$-class of the spectral triple is then the element
$$
[\rho]\otimes_{A}[\boldsymbol\pi_B,\boldsymbol{B}_B,\slashed{\boldsymbol{R}}]=[\boldsymbol\pi_B\circ\rho,\boldsymbol{B}_B,\slashed{\boldsymbol{R}}]
$$ 
in $KK^0(\C,B)$. The map $\boldsymbol\pi_B\circ\rho$ sends $1\in\C$ to the operator $\boldsymbol\pi_B\circ\rho(1)=\boldsymbol\pi_B(e)-\boldsymbol\pi_B(f)$. 
From 
$$
(\boldsymbol\pi_B(e)\slashed{\boldsymbol{R}}_+\boldsymbol\pi_B(e))(\boldsymbol\pi_B(e)\slashed{\boldsymbol{R}}_-\boldsymbol\pi_B(e))=\boldsymbol\pi_B(e)\quad \text{mod }B\otimes\Ki,
$$
$$
(\boldsymbol\pi_B(e)\slashed{\boldsymbol{R}}_-\boldsymbol\pi_B(e))(\boldsymbol\pi_B(e)\slashed{\boldsymbol{R}}_+\boldsymbol\pi_B(e))=\boldsymbol\pi_B(e)\quad \text{mod }B\otimes\Ki,
$$
we see that $\boldsymbol\pi_B(e)\slashed{\boldsymbol{R}}_+\boldsymbol\pi_B(e)$ is Fredholm as an operator from the module $\boldsymbol\pi_B(e)\boldsymbol{B}_+$ to the module $\boldsymbol\pi_B(e)\boldsymbol{B}_-$.

We identify $q(\boldsymbol\pi_B(e)\slashed{\boldsymbol{R}}_+\boldsymbol\pi_B(e))$ with a unitary in $\Qi(B\otimes\Ki)$ and similarly with $f$ replacing $e$. The difference $q(\boldsymbol\pi_B(e)\slashed{\boldsymbol{R}}_+\boldsymbol\pi_B(e))
-q(\boldsymbol\pi_B(f)\slashed{\boldsymbol{R}}_+\boldsymbol\pi_B(f))$ represents the class $[\rho]\otimes_{A}[\boldsymbol\pi_B,\boldsymbol{B}_B,\slashed{\boldsymbol{R}}]$ under the identification of $KK^0(\C,B)$ with $K_1(\Qi(B\otimes\Ki))$. 

The isomorphism $\delta:K_1(\Qi(B\otimes\Ki))\to K_0(B)$ just sends the homotopy class $[q(\boldsymbol\pi_B(e)\slashed{\boldsymbol{R}}_+\boldsymbol\pi_B(e))-q(\boldsymbol\pi_B(f)\slashed{\boldsymbol{R}}_+\boldsymbol\pi_B(f))]$ to the $K_0(B)$-valued index of $\boldsymbol\pi_B(e)\slashed{\boldsymbol{R}}_+\boldsymbol\pi_B(e)-\boldsymbol\pi_B(f)\slashed{\boldsymbol{R}}_+\boldsymbol\pi_B(f)$, as asserted.
\end{proof}

\begin{Remark}[The obstruction to using $\slashed{F}$]
In general we cannot use $\slashed{F}:=\slashed{D}(1+\slashed{D}^2)^{-1/2}$ instead of $\slashed{\boldsymbol{R}}$ in the above pairings. The problem is that $e\slashed{F}_+e$ need not be Fredholm when $A$ is nonunital and $\slashed{F}^2\ne \bone$. We recall the details about this fact from \cite[\S 2.3]{CGRS1}.

Generally, let $A$ and $B$ be $C^*$-algebras and let $(\pi_B,X_B,F)$ be an even Kasparov $A$-$B$ module. Let $e\in A^\sim$ be a projection. We would like $\pi_B(e)F_+\pi_B(e)$ to be a Fredholm operator from $\pi_B(e)X_+$ to $\pi_B(e)X_-$. So we try to show that $\pi_B(e)F_+\pi_B(e)$ is invertible modulo $\Ki_B(X)$. We have
$$
(\pi_B(e)F_-\pi_B(e))(\pi_B(e)F_+\pi_B(e))=\pi_B(e)F_-[\pi_B(e),F_+]\pi_B(e)+\pi_B(e)(F_-F_+-\bone)\pi_B(e)+\pi_B(e).
$$ 
The term $\pi_B(e)F_-[\pi_B(e),F_+]\pi_B(e)$ is compact. Indeed $[\pi_B(a),F_+]$ was required to be compact for $a\in A$ by definition of Kasparov module, and elements of the form $\lambda\bone\in A^\sim$ have trivial commutators. The problematic term is $\pi_B(e)(F_-F_+-\bone)\pi_B(e)$, which is guaranteed to be compact only for $e\in A$. 

This is where the condition $F^2=\bone$ becomes important. If $F^2=\bone$ then $F_-F_+-\bone=0$ and so $\pi_B(e)F_+\pi_B(e)$ is Fredholm. 
\end{Remark}

\subsection{Numerical index}

We saw in Theorem \ref{oddKaspprod} that for any $[u]\in K_1(A)$, the element $T_u:=\slashed{\boldsymbol{P}}\boldsymbol{\pi}_B(u)\slashed{\boldsymbol{P}}\in\Mi(B\otimes\Ki)$ is a Fredholm operator on the Hilbert $B$-module $\ell^2(\N;B)$. The projections $\Ker(T_u)$ and $\Ker(T_u^*)$ are of finite rank and can be regarded as elements of $B\otimes\Ki$. Now the dual trace $\hat{\tau}$ induces a homomorphism $\hat{\tau}_*:K_0(B)\to\R$. 
Therefore, the Fredholm operator $T_u$ is $(\hat{\tau}\otimes\Tr)$-Fredholm in the ``semifinite" sense of \cite{Az1, BeFa1, Breu1, Breu2, CGRS1}, where $\Tr$ is operator trace on $\Ki$. 

Similarly, if $(\pi_\alpha,\Mn_N(B),\slashed{F})$ is even then the Thom class produces $(\hat{\tau}\otimes\Tr)$-Fredholm operators from elements of $K_0(A)$. 

\begin{dfn}[{cf. \cite[Def. 2.12]{CGRS1}}]
The \textbf{index pairing} of $[x]\in K_\bullet(A)$ with the Thom class $\mathbf{t}_\alpha=[\pi_\alpha,\Mn_N(B),\slashed{F}]\in KK^\bullet(A,B)$ is the real number 
$$
\bra[x],\mathbf{t}_\alpha\ket:=\hat{\tau}_*([x]\otimes_A\mathbf{t}_\alpha)
$$
obtained by applying the homomorphism $\hat{\tau}_*:K_0(B)\to\R$ to the Kasparov product $[x]\otimes_A\mathbf{t}_\alpha$. 
\end{dfn}

\begin{cor}\label{oddandevenFredpair} 
Suppose that $n$ is odd. Let $u\in\Un_\infty(\Ai^\sim)$ be a unitary over $\Ai^\sim$ and denote by $[u]\in K_1(A)$ the homotopy class of $u$. Then we have the equality
$$
\bra[u],\mathbf{t}_\alpha\ket=\Index_{\hat{\tau}}(\slashed{\boldsymbol{P}}\boldsymbol\pi_B(u)\slashed{\boldsymbol{P}})
$$
in $\hat{\tau}_*(K_0(B))\subseteq\R$.

Suppose that $n$ is even. Let $e,f\in\Mn_\infty(\Ai^\sim)$ be projections over $\Ai^\sim$ with $[e]-[f]\in K_0(A)$. Then we have the equality
$$
\bra[e]-[f],\mathbf{t}_\alpha\ket=\Index_{\hat{\tau}}(\boldsymbol\pi_B(e)\slashed{\boldsymbol{R}}_+\boldsymbol\pi_B(e))
-\Index_{\hat{\tau}}(\boldsymbol\pi_B(f)\slashed{\boldsymbol{R}}_+\boldsymbol\pi_B(f))
$$
in $\hat{\tau}_*(K_0(B))\subseteq\R$.
\end{cor}

\subsection{Local formula}\label{numer}
In this section we will prove a formula in the spirit of \cite{CGPRS, Le, PR} for the $\hat{\tau}$-index of Toeplitz operators $T_u=P\pi_\alpha(u)P$. We adopt the powerful approach to the case $n=1$ given in the recent paper \cite{CGPRS}. In particular, we will need the local index formula for nonunital semifinite spectral triples described in \cite[\S3]{CGRS1}.

We want to apply the general version of the local index formula to compute the $\hat{\tau}$-index. For that, we first of all need to find a nonzero $*$-algebra $\Ci\subset A$ which gives a smoothly summable $(\Ni,\hat{\tau})$-semifinite spectral triple. Recall that for smooth summability \cite[\S3]{CGRS1} we need both a suitable smoothness property of elements $a$ in $\Ci$ with respect $\slashed{D}$, as well as a $(\hat{\tau},\slashed{D})$-integrability condition on $\pi_\alpha(a)$. 

As expected, in our setting the smoothness with respect to $\slashed{D}$ is tightly related to the smoothness with respect to the $\R^n$-action $\alpha$. In fact, we shall obtain as in \cite[Prop. 3.12]{CGPRS} that if an element $a\in A$ is smooth for the generator $\delta$ of $\alpha$ then $\pi_\alpha(a)$ is smooth for the derivation $[|\slashed{D}|,\cdot]$. Hence the $\alpha$-smooth subalgebra $\Ai\subset A$ can be used to define a smoothly summable spectral triple $(\Ai,\Hi,\slashed{D})$. We will not be able to use all of $\Ai$ since we also need an integrability condition, but we will still be able to find a dense $*$-subalgebra $\Ci\subset\Ai$ of $A$ such that the inclusion $\Ci\hookrightarrow A$ induces an isomorphism on $K$-theory. 

First we shall discuss how integrability properties will be affected by the choice of Hilbert space. Remember that $A$ is acting on a Hilbert space $\GH$ and that
$$
\pi_\alpha:A\to L^2(\R^n,\GH)
$$
is defined in terms of $\GH$. The dual trace $\hat{\tau}$ on $\Ni$ can be alternatively defined \cite[Def. 3.1]{H1}, \cite[Def. X.1.6]{Ta2} in terms of the Hilbert algebra
\begin{equation}\label{Hilbalgofdual}
\GA_\tau:=L^2(\R^n,\GH_\tau)\cap L^1(\R^n,\Dom(\tau)),
\end{equation}
where $\GH_\tau$ is the GNS space of $\tau$. It is therefore natural to want $\pi_\alpha$ to be a representation on $L^2(\R^n,\GH_\tau)$, and this was the approach in \cite{CGPRS}. The action $\alpha$ is then required to preserve the trace, or else it will not have a unitary implementation. However, the dual trace can also be described (see Definition \ref{dualtrace}) as the composition of $\tau$, $\pi^{-1}_\alpha$ and the operator-valued weight $E$ in \eqref{opvaluedweight} and as we shall see, we do not need $\GH$ to be $\GH_\tau$. (Again, we \emph{do} assume that $\alpha$ preserves $\tau$ in this work but we aim for some flexibility in the choice of $\GH$ that could be useful in the future.) The reason for this is the isomorphism $\Ni\cong\hat{\pi}_\alpha(\GA_\tau)''$ \cite[Lemma X.1.15]{Ta2}, where $\GA_\tau$ is the left Hilbert algebra \eqref{Hilbalgofdual} which completely defines $\hat{\tau}$.

Again we use $\pi_\alpha(a)$ to denote $1_N\otimes\pi_\alpha(a)$ for $a\in A$ and we write $\hat{\tau}$ for $\Tr\otimes\hat{\tau}$ where $\Tr$ is the matrix trace on $\Mn_N(\C)$.

Let $\delta_\tau$ denote the restriction of the generators $\delta=(\delta_1,\dots,\delta_n)$ to $\Dom(\tau)$ and let $\Dom(\delta_\tau)\subset\Dom(\tau)$ denote the domain of $\delta_\tau$.
\begin{Lemma}
The triple $(\Dom(\delta_\tau),\Hi,\slashed{D})$ is a \textbf{$(\Ni,\hat{\tau})$-semifinite spectral triple} over $A$, i.e. for all $a\in\Dom(\delta_\tau)$ it holds that \cite[Def. 2.1]{CGRS1}
\begin{enumerate}[(i)]
\item{the operator $\pi_\alpha(a)$ preserves the domain of $\slashed{D}$ (implying that the commutator $[\slashed{D},\pi_\alpha(a)]$ is densely defined), $[\slashed{D},\pi_\alpha(a)]$ extends to a bounded operator on $\Hi$, and}
\item{$\pi_\alpha(a)(\bone+\slashed{D}^2)^{-1/2}$ belongs to the ideal $\Ki(\Ni,\hat{\tau})$ of $\hat{\tau}$-compact operators.}
\end{enumerate}
\end{Lemma}
\begin{proof} Property (i) is obvious since $\Dom(\delta_\tau)\subset\Ai$. For (ii) we note that $\pi_\alpha(a)(\bone+\slashed{D}^2)^{-1/2}=\tilde{\pi}_\alpha(f)$ where the function $\R^n\ni t\to f(t):=a(1+|t|^2)^{-1/2}$ belongs to $\Dom(\delta_\tau)\otimes C_0(\R^n)$.
\end{proof}

Let $A''$ be the weak closure of $A$ in its original representation. Then $\tau$ extends to a normal trace $\bar{\tau}$ on $A''$ with the same GNS space as $\tau$. The following lemma is the counterpart of \cite[Lemma 3.4]{CGPRS}.

\begin{Lemma}\label{traceformulalemma} 
Let $h\in L^\infty(\R^n)\cap L^2(\R^n)$ and let $a\in A''$ be such that $a^*a$ is in $\Dom(\bar{\tau})$. If we define
$$
x(t):=ah(t),
$$ 
then $\hat{\pi}_\alpha(x)\in\Ni$ is $\hat{\tau}$-Hilbert-Schmidt and
$$
\hat{\tau}(\hat{\pi}_\alpha(x)^*\hat{\pi}_\alpha(x))=\tau(a^*a)\int_{\R^n}|h(t)|^2\, dt.
$$
\end{Lemma}
\begin{proof} We write $\hat{\pi}_\alpha(x)=\int \pi_\alpha(a)h(s)e^{-2\pi is\cdot D}\, ds$ so that
$$
\hat{\alpha}_p(\hat{\pi}_\alpha(x))=\int_{\R^n}\pi_\alpha(a)h(s)e^{-2\pi ip\cdot s}e^{-is\cdot D}\, ds,\qquad\forall p\in\R^n.
$$
Since $\hat{\tau}=\bar{\tau}\circ\pi^{-1}_\alpha\circ E$, the assumptions on $x$ give
\begin{align*}
\hat{\tau}(\hat{\pi}_\alpha(x)^*\hat{\pi}_\alpha(x))&=\bar{\tau}\circ\pi^{-1}_\alpha\Big(\int_{\R^n}\hat {\alpha}_p(\hat{\pi}_\alpha(x^*x))\, dp\Big)
\\&=\bar{\tau}\circ\pi^{-1}_\alpha\Big(\iiint_{\R^n\times\R^n\times\R^n}\pi_\alpha(a^*a)\overline{h(t)}h(s+t)e^{-2\pi ip\cdot s}e^{-2\pi is\cdot D}\, dt\, ds\, dp\Big)
\\&=\bar{\tau}(a^*a)\int_{\R^n}|h(t)|^2\, dt.
\end{align*}
\end{proof}

\begin{cor}\label{varphiproptotau}
Let $s>n$ and define a weight $\varphi_s$ on $\Ni$ by setting
$$
\varphi_s(T):=\hat{\tau}((1+\slashed{D}^2)^{-s/4}T(1+\slashed{D}^2)^{-s/4})
$$
for all $T\in\Ni_+$. Then the restriction of $\varphi_s$ to $\Mi:=\pi_\alpha(A)''$, viewed as a subalgebra of $\Ni$, is proportional to $\bar{\tau}\circ\pi^{-1}_\alpha$.
\end{cor}
\begin{proof} From the Clifford relations we get
$$
\slashed{D}^2=\sum^n_{k=1}\bone\otimes D_k^2,
$$
and so if $h_s(t):=(1+|t|^2)^{-s/4}$ then by Lemma \ref{traceformulalemma} we have for each positive $a$ in the domain of $\tau$ that
\begin{align*}
\varphi_s(\pi_\alpha(a))&=\hat{\tau}((1+\slashed{D}^2)^{-s/4}\pi_\alpha(a)(1+\slashed{D}^2)^{-s/4})
\\&=\hat{\tau}(h_s(\slashed{D})\pi_\alpha(a)h_s(\slashed{D}))
\\&=\Tr(\bone_N)\tau(a)\int_{\R^n}|h_s(t)|^2\, dt.
\end{align*}
It follows that more generally that $\varphi_s=\|h_s\|^2_2\bar{\tau}\circ\pi^{-1}_\alpha$ holds on $\pi_\alpha(\Dom(\bar{\tau})_+)\subset\Mi_+$. That $\varphi_s(\pi_\alpha(a))=+\infty$ whenever $\bar{\tau}(a)=+\infty$ can be seen as in Corollary 3.5 of \cite{CGPRS}.
\end{proof}
We now need the notion of ``smooth summability". Recall that the \textbf{half-domain} of a weight $\varphi:\Ni_+\to[0,+\infty]$ is the vector space
$$
\Dom^{1/2}(\varphi):=\{T\in \Ni|\ \varphi(T^*T)<+\infty\}.
$$
\begin{dfn}[{\cite[\S 1.1]{CGRS1}}]
Let $p\geq 1$ be a real number. The algebra of \textbf{$(\slashed{D},\hat{\tau},p)$-square-integrable} elements in $\Ni$ is the one defined by
$$
\Bi_2(\slashed{D},\hat{\tau},p):=\bigcap_{s>p}\Dom^{1/2}(\varphi_s)\cap \Dom^{1/2}(\varphi_s)^*.
$$
For each $p\geq 1$, the space $\Bi_2(\slashed{D},\hat{\tau},p)$ is a Fr\'{e}chet $*$-algebra \cite[Prop. 1.6]{CGRS1}. The algebra of \textbf{$(\slashed{D},\hat{\tau},p)$-integrable} elements in $\Ni$ is the subalgebra
$$
\Bi_1(\slashed{D},\hat{\tau},p)\subset\Bi_2(\slashed{D},\hat{\tau},p)
$$
defined as the closure of the image of $\Bi_2(\slashed{D},\hat{\tau},p)\otimes\Bi_2(\slashed{D},\hat{\tau},p)$ (completed projective tensor product) under the multiplication map. 
\end{dfn}
\begin{dfn}\label{quantumsmooth} 
Consider the unbounded operators $L$ on $\Ni$ given by
$$
L(T):=(1+\slashed{D}^2)^{-1/2}[\slashed{D}^2,T],\qquad\forall T\in\Ni.
$$
A $(\Ni,\hat{\tau})$-semifinite spectral triple $(\Ai,\Hi,\slashed{D})$ is \textbf{smoothly summable} if there is a $p\geq 1$ such that $\pi(a)$ and $[\slashed{D},\pi(a)]$ belong to the algebra \cite[Lemma 1.29]{CGRS1}
$$
\Bi^\infty_1(\slashed{D},\hat{\tau},p):=\{T\in \Bi_1(\slashed{D},\hat{\tau},p)|\ L^k(T)\in\Bi_1(\slashed{D},\hat{\tau},p)\text{ for all }k\in\N\}.
$$
for all $a\in\Ai$. 
\end{dfn}

\begin{Lemma}\label{Cissmoothsumm}
Let $\Ci$ be the $*$-subalgebra of $\Dom(\tau)$ generated by the set
\begin{equation}\label{Cisthisagain}
\{a=bc\in\Ai|\ \delta^k(b),\, \delta^k(c)\in\Dom(\delta_\tau)\text{ for all }k\in\N_0\}.
\end{equation}
Then $(\Ci,\Hi,\slashed{D})$ is a smoothly summable $(\Ni,\hat{\tau})$-semifinite spectral triple. That is, there is a $p\geq 1$ such that $\pi_\alpha(a)$ and $[\slashed{D},\pi_\alpha(a)]$ belong to $\Bi^\infty_1(\slashed{D},\hat{\tau},p)$ for all $a\in\Ci$. 
\end{Lemma}
\begin{proof} 
We anticipate that $p=n$ will suffice. As in Definition \ref{quantumsmooth}, let $L$ be the operator on $\Ni$ given by $L(T):=(1+\slashed{D}^2)^{-1/2}[\slashed{D}^2,T]$. We begin by showing that $\pi_\alpha(\Ai)$ is contained in the smooth domain of $L$. 

Write $\slashed{F}:=\slashed{D}(1+\slashed{D}^2)^{-1/2}=(1+\slashed{D}^2)^{-1/2}\slashed{D}$ and $F_k:=D_k(1+\slashed{D}^2)^{-1/2}$ for $k=1,\dots,n$. Using Proposition \ref{easycommutator} and  we see that
\begin{align*}
L(\pi_\alpha(a))&=(1+\slashed{D}^2)^{-1/2}(\slashed{D}[\slashed{D},\pi_\alpha(a)]+[\slashed{D},\pi_\alpha(a)]\slashed{D})
\\&=\frac{1}{2\pi i}\sum^n_{k=1}\slashed{F}\gamma^k\pi_\alpha(\delta_k(a))+\frac{1}{2\pi i}(1+\slashed{D}^2)^{-1/2}\sum^n_{k=1}\gamma^k\pi_\alpha(\delta_k(a))\slashed{D}
\\&=\frac{1}{2\pi i}\sum^n_{k=1}\slashed{F}\gamma^k\pi_\alpha(\delta_k(a))+\frac{1}{2\pi i}(1+\slashed{D}^2)^{-1/2}\sum^n_{k=1}\big([\gamma^k\pi_\alpha(\delta_k(a)),\slashed{D}]+\slashed{D}\gamma^k\pi_\alpha(\delta_k(a))\big)
\\&=\frac{1}{\pi i}\sum^n_{k=1}\slashed{F}\gamma^k\pi_\alpha(\delta_k(a))+\frac{1}{2\pi i}(1+\slashed{D}^2)^{-1/2}\sum^n_{l,k=1}\Big(\frac{1}{2\pi i}\gamma^l\gamma^k\pi_\alpha(\delta_l\delta_k(a))+[\gamma^l,\gamma^k]D_k\pi_\alpha(\delta_k(a))\Big)
\\&=\frac{1}{2\pi i}\sum^n_{k=1}\big(2\slashed{F}\gamma^k+[\gamma^l,\gamma^k]F_k\big)\pi_\alpha(\delta_k(a))+\frac{1}{4\pi^2}(1+\slashed{D}^2)^{-1/2}\sum^n_{l,k=1}\gamma^l\gamma^k\pi_\alpha(\delta_l\delta_k(a))
\end{align*}
whenever $a$ belongs to the domain of $\delta^2$. 

This shows that $\Dom(L)\subset\pi_\alpha(\Dom(\delta^2))$. Since $L^j$ is defined using the derivation $[\slashed{D}^2,\cdot]$, for all $f,g\in L^\infty(\R^n)$ and $T\in\Dom(L^j)$ one has
$$
L^j(f(D)Tg(D))=f(D)L^j(T)g(D).
$$
Note that $L(\pi_\alpha(a))$ is of the form $f(D)Tg(D)$ with $f,g\in L^\infty(\R^n)$ and $T\in\Dom(L)$. Thus, the action of $L^j$ for $j\in\N$ can be deduced by just repeating the above calculation with $\delta_k(a)$ and $\delta_l\delta_k(a)$ instead of $a$, provided that $a$ belongs to the domain of $\delta^{2j}$. Thus, an element $a$ in $\bigcap_{r\in\N}\Dom(\delta^r)=\Ai$ will belong to $\Dom(L^j)$ for all $j\in\N$. 

From Corollary \ref{varphiproptotau} we have 
$$
\pi_\alpha(\Dom^{1/2}(\tau))\subset\Bi_2(\slashed{D},\hat{\tau},n).
$$
From \cite[Prop. 1.19]{CGRS1} we know that, since each $\varphi_s$ is tracial on $\Mi:=\pi_\alpha(A)''$, the space $\Bi_1(\slashed{D},\hat{\tau},n)\cap\Mi$ is equal to the intersection of trace-ideals $\Li^1(\Mi,\varphi_s)=\Dom(\varphi_s)$,
$$
\Bi_1(\slashed{D},\hat{\tau},n)\cap\Mi=\bigcap_{s>n}\Li^1(\Mi,\varphi_s).
$$ 
But $\varphi_s$ is proportional to $\bar{\tau}\circ\pi^{-1}_\alpha$ on $\Mi$ for all $s>n$, so we obtain
$$
\Bi_1(\slashed{D},\hat{\tau},n)\cap\Mi=\pi_\alpha(\Dom(\bar{\tau})).
$$
Finally, on the $C^*$-level this yields
$$
\Bi_1(\slashed{D},\hat{\tau},n)\cap\pi_\alpha(A)=\pi_\alpha(\Dom(\tau)).
$$
In particular, $\pi_\alpha(\Ci)\subset\Bi_1(\slashed{D},\hat{\tau},n)$. From the definition of $\Ci$ and our calculation of $L(\pi_\alpha(\Ci))$ it follows that $L^k(\pi_\alpha(\Ci))\subset\Bi_1(\slashed{D},\hat{\tau},n)$ for all $k\in\N$. Moreover, we have $[\slashed{D},\pi_\alpha(a)]\in\Mn_N(\C)\otimes\pi_\alpha(\Ci)$ for all $a\in\Ci$ (Proposition \ref{easycommutator}). Thus, $L^k([\slashed{D},\pi_\alpha(\Ci)])\subset\Bi_1(\slashed{D},\hat{\tau},n)$ for all $k\in\N$ as well.
That finishes the proof.
\end{proof}
\begin{cor}\label{traceclassprop} 
For all $a\in\Ci$ we have 
$$
\pi_\alpha(a)(1+\slashed{\boldsymbol D}^2)^{-s/2}\in\Li^1(\Ni,\hat{\tau}),\qquad\forall s>n.
$$
\end{cor}
\begin{proof} As we have seen, for $a\in\Ci$ we have $\pi_\alpha(a)\in\Bi^\infty_1(\slashed{D},\hat{\tau},n)$. Therefore, for all $s$ the operator $\pi_\alpha(a)(1+\slashed{\boldsymbol D}^2)^{-s/2}$ belongs to the space of pseudodifferential operators of order $-s$ denoted by $\OP^{-s}_0(\slashed{D},\hat{\tau},n)$ in \cite[\S1.4]{CGRS1}. One has $\OP^{0}_0(\slashed{D},\hat{\tau},n)=\Bi^\infty_1(\slashed{D},\hat{\tau},n)$, so $\OP^{-s}_0(\slashed{D},\hat{\tau},n)\subset\Li^1(\Ni,\hat{\tau})$ for $s>n$ by \cite[Cor. 1.30]{CGRS1}. 
\end{proof}

In the following we use the notation \eqref{shrthand}.
\begin{prop}\label{finalindex} For $n$ odd and a unitary $u\in \Ci^\sim$, we have
$$
\Index_{\hat{\tau}}(T_u)=-\frac{2^{(n-1)/2}(-1)^{(n-1)/2}((n-1)/2)!}{(2\pi i)^nn!}\tau\big((u^*\delta(u))^n\big).
$$
\end{prop}
\begin{proof} The proof of Lemma \ref{Cissmoothsumm} shows that $(\Ci,\Hi,\slashed{D})$ has spectral dimension $n$. 

From the odd part of the local index formula \cite[Thm. 3.33]{CGRS1} in term of the resolvent cocycle, we have
\begin{equation}\label{reolvshoulddeg}
\Index_{\hat{\tau}}(T_u)=\frac{-1}{\sqrt{2\pi i}}\ \underset{r=(1-n)/2}{\Res}\sum^{n}_{m=1,\textnormal{ odd}}\Phi_m^r(\Ch_m(u)).
\end{equation}
Here $\Ch_m(u)$ is the Chern character of $[u]$,
$$
\Ch_m(u):=(-1)^{(m-1)/2}((m-1)/2)!\, u^{-1}\otimes u\otimes\cdots\otimes u^{-1}\otimes u\in(\Ci^\sim)^{\otimes (m+1)},
$$
and $\Phi^r_m(\Ch_m(u))$ is, up to some constants (see \cite[Def. 3.4]{CGRS1}), the integral over $s\in\R_+$ of the function 
$$
(\Tr\otimes\hat{\tau})\Big(\frac{1}{2\pi i}s^m\Gamma\int_{\ep+i\R}\lambda^{-n/2-r}\pi_\alpha(u^{-1})R_s(\lambda)[\slashed{D},\pi_\alpha(u)]R_s(\lambda)\cdots[\slashed{D},\pi_\alpha(u)]R_s(\lambda)\, d\lambda\Big)
$$
where $R_s(\lambda):=(\lambda-(\bone+s^2+\slashed{D}^2))^{-1}$. There is a product of $m$ commutators $[\slashed{D},\pi_\alpha(u)]=(2\pi i)^{-1}\sum_k\gamma^k[D_k,\pi_\alpha(u)]$ in the above expression, and hence a factor $\Tr(\Gamma\gamma^{k_1}\cdots\gamma^{k_m})$. Only a product of $n$ Clifford generators $\gamma^k$ has nonzero graded trace \cite[Prop. 3.21]{BGV} (note that $\Gamma=\bone$ here because $n$ is odd, but the mentioned fact is true for even $n$ as well). Therefore, only the $n$th component in right-hand side of \eqref{reolvshoulddeg} survives. Theorem 3.33 of \cite{CGRS1} says that the function $r\to \Phi^r_n(\Ch_n(u))$ can be analytically continued to a deleted neighborhood of $r=(1-n)/2$ where it has at worst a simple pole. 

The fact that only one term $\Phi^r_n(\Ch_n(u))$ survives and has a well-defined residue at $r=(1-n)/2$ allows the proof of \cite[Prop. 3.20]{CGRS1} to be carried out without the hypothesis of isolated spectral dimension. The result is that $\Index_{\tau}(T_u)$ equals $-(2\pi i)^{-1/2}$ times the value of the residue cocycle
$$
\phi_n(a_0,a_1,\dots a_n):=\frac{\sqrt{2\pi i}}{n!}\,\underset{s=n}{\Res}\ \hat{\tau}(\boldsymbol\pi_\alpha(a_0)[\slashed{\boldsymbol D},\boldsymbol\pi_\alpha(a_1)]\cdots[\slashed{\boldsymbol D},\boldsymbol\pi_\alpha(a_n)](1+\slashed{\boldsymbol D}^2)^{-n/2-s})
$$  
on the cycle $\Ch_n(u)$. That is,
\begin{align*}
\Index_{\hat{\tau}}(T_u)&=-\frac{(-1)^{(n-1)/2}((n-1)/2)!}{\sqrt{2\pi i}}\, \phi_n(u^*,u,\dots,u^*,u),
\end{align*}
Recall the explicit expression for the commutators $[\slashed{D},\pi_\alpha(a)]$ from Proposition \ref{easycommutator}. In the notation \eqref{shrthand}, Lemma \ref{traceformulalemma} shows that 
$$
\phi_n(u^*,u,\dots,u^*,u)=\frac{2^{(n-1)/2}\sqrt{2\pi i}}{(2\pi i)^nn!}\tau\big(u^*(\delta(u)\delta(u^*))^{(n-1)/2}\delta(u)\big)\underset{s=n}{\Res}\int_{\R^n}(1+|t|^2)^{-s/2}\, dt
$$
(the factors of $1/2\pi i$ come from Proposition \ref{easycommutator} while the factor $2^{(n-1)/2}$ is the trace of the product of all $\gamma$ matrices). Since $\delta$ is a derivation, $uu^{-1}=\bone$ gives
$$
\delta(u^{-1})=-u^{-1}\delta(u)u^{-1},
$$
and so
\begin{align*}u^*\big((\delta(u)\delta(u^*)\big)^{(n-1)/2}\delta(u)&=u^*\big(\delta(u^*)u^*\delta(u)u^*\big)^{(n-1)/2}\delta(u)=(u^*\delta(u))^n,
\end{align*}
from which
$$
\phi_n(u^*,u,\dots,u^*,u)=\frac{2^{(n-1)/2}\sqrt{2\pi i}}{(2\pi i)^nn!}\tau\big((u^*\delta(u))^n\big)\underset{s=n}{\Res}\int_{\R^n}(1+|t|^2)^{-s/2}\, dt.
$$
\end{proof}
 
\begin{prop}\label{finalindexeven} For $n$ even and a projection $e\in\Ci^\sim$, we have
$$
\Index_{\hat{\tau}}(\boldsymbol\pi_\alpha(e)\slashed{\boldsymbol R}_+\boldsymbol\pi_\alpha(e))=\frac{(-1)^{n/2}}{(n/2)!}\frac{2^n}{(2\pi i)^n}\tau\big((e\delta(e)\delta(e))^{n/2}\big).
$$
\end{prop}
\begin{proof} For the same reason as in Proposition \ref{finalindex}, we obtain the relation
$$
\Index_{\hat{\tau}}(\boldsymbol\pi_\alpha(e)\slashed{\boldsymbol R}_+\boldsymbol\pi_\alpha(e))=\phi_0(e)+\frac{(-1)^{n/2}n!}{(n/2)!}\phi_n(e,\dots,e).
$$
The zeroth term
$$
\phi_0(e)=\underset{z=0}{\Res}\,\frac{1}{z}\Tr(\boldsymbol\Gamma\boldsymbol\pi_\alpha(e)(\bone+\slashed{\boldsymbol D}^2)^{-z})
$$
is $0$ because the grading $\Gamma=\diag(\bone,-\bone)$ gives $\Tr(\boldsymbol\Gamma\boldsymbol\pi_\alpha(e)(\bone+\slashed{\boldsymbol D}^2)^{-z})=0$. Now the expression from Lemma \ref{traceformulalemma},
$$
\phi_n(e,\dots,e)=\frac{2^n}{(2\pi i)^nn!}\tau\big(e\delta(e)\cdots\delta(e)\big)\underset{s=n}{\Res}\int_{\R^n}(1+|t|^2)^{-s/2}\, dt,
$$
can rearranged using $e(\delta(e))^{n-1}=(e\delta(e)\delta(e))^{n/2}$, which follows from idempotency of $e$. 
\end{proof}
By \cite[Prop. 2.20]{CGRS1}, the completion $\Ci_{\delta,\varphi}$ of $\Ci$ in a certain locally convex topology is a dense $*$-subalgebra of $A$ such that the inclusion $\Ci_{\delta,\varphi}\hookrightarrow A$ induces isomorphisms on both $K$-groups and $(\Ci_{\delta,\varphi},\Hi,\slashed{D})$ is again a smoothly summable spectral triple over $A$. Therefore, given \emph{any class} $[x]\in K_\bullet(A)$ there is a representative $x\in\Ci_{\delta,\varphi}$ such that a matrix analogue of one of the formulas (depending on the parity $\bullet$ of $n$) in Proposition \ref{finalindex} or Proposition \ref{finalindexeven} holds.

\subsection{Another choice of projection}
We now construct a Toeplitz extension without doubling up the Hilbert space. We shall use the same notation throughout, since it will be clear from the context which of the Toeplitz extensions is considered. Recall that $\Hi:=\C^N\otimes L^2(\R^n,\GH)$.

Let $\Ti$ be the $C^*$-subalgebra of $\Bi(\Hi)$ generated by $\Mn_N(B)$ and the Toeplitz operators
$$
T_a:=\slashed{P}\pi_\alpha(a)\slashed{P},\qquad a\in A,
$$
where $\slashed{P}$ is the spectral projection onto the nonnegative part of the spectrum of the massless Dirac operator $\slashed{D}$. 
\begin{prop} For $n$ odd, there is a semi-split extension 
$$
0\longto \Mn_N(B)\longto\Ti\longto A\longto 0.
$$
The triple $(\pi_\alpha,\Mn_N(B),2\slashed{P}-\bone)$ is a Kasparov $A$-$B$ module representing the same class as the double $(\boldsymbol\pi_\alpha,\boldsymbol\Mn_N(B),\slashed{\boldsymbol R})$. In particular, for all $[u]\in K_1(A)$ we have
$$
\Index(\slashed{ P}\pi_\alpha(u)\slashed{P})=\Index(\slashed{\boldsymbol P}\boldsymbol\pi_\alpha(u)\slashed{\boldsymbol P})
$$
and for all $[e]\in K_0(A)$ we have
$$
\Index(\pi_\alpha(e)(2\slashed{P}-\bone)_+\pi_\alpha(e))=\Index(\boldsymbol\pi_\alpha(e)\slashed{\boldsymbol R}_+\boldsymbol\pi_\alpha(e)).
$$
\end{prop}
\begin{proof} By \cite[Prop. 2.25]{CGRS1}, $(\pi_\alpha,\Mn_N(B),2\slashed{P}-\bone)$ is a Kasparov $A$-$B$ module representing the same $KK$-class as $(\pi_\alpha,\Mn_N(B),\slashed{F})$. The extension associated with $(\pi_\alpha,\Mn_N(B),2\slashed{P}-\bone)$ under the isomorphism $KK^1(A,B)\cong \Ext(A,B)$ gives the Toeplitz extension in the statement.
\end{proof}
Recall that the $KK$-classes of $(\boldsymbol\pi_\alpha,\Mn_N(\boldsymbol B),\slashed{\boldsymbol R})$ and $(\pi_\alpha,\Mn_N(B),\slashed{F})$ coincide. Therefore, $\slashed{P}$ and $2\slashed{F}-\bone$ also defined the same element in $KK^\bullet(A,B)$, where $\bullet\in\{0,1\}$ is equal to $0$ if $n$ is even and equal to $1$ if $n$ is odd.

\section{Rieffel deformations}
With an action $\alpha$ of $\R^n$ on a $C^*$-algebra $A$ and a skew-symmetric $n\times n$ matrix $\Theta$ one can consider the Rieffel deformation $A_\Theta$ of $A$, which is a $C^*$-algebra generated by the $\alpha$-smooth elements of $A$ with a new multiplication. The purpose of this section is to obtain explicit formulas for $K$-theoretical quantities defined by elements of $A_\Theta$. Our approach relies on the smoothly summable spectral triple $(\Ci,\Hi,\slashed{D})$ associated with the $C^*$-dynamical system $(A,\R^n,\alpha)$ as in the last section.

Rieffel showed that the $K$-theories of $A$ and $A_\Theta$ are isomorphic \cite{Rie5}. However, there is no explicit description for the generators of $K_\bullet(A_\Theta)$ even when the generators of $K_\bullet(A)$ are known. A projection $e\in\Mn_\infty(A_\Theta)$ (so that $e\times_\Theta e=e)$ need not be a projection in $\Mn_\infty(A)$ (i.e. $e^2=e$ may not hold) and vice versa. Similar remarks hold for unitaries.  

We shall use three different pictures of Rieffel deformation to describe the relation between the index pairings for $A$ and $A_\Theta$.

\subsection{Rieffel deformations in three ways}\label{deform}

\subsubsection{Quantization of noncommutative algebras}
Motivated by the mathematical theory of quantization, Rieffel introduced a way of deforming a $C^*$- or Fr\'{e}chet algebra by changing the multiplication \cite{Rie1, Rie3}. He shows that the resulting algebra is a $C^*$-algebra and that the construction is functorial in a certain sense. His approach is very analytical and technical, based on the use of operator-valued oscillatory integrals. 


Let $A$ be a $C^*$-algebra and let $C_u(\R^n,A)$ be the $C^*$-algebra of bounded uniformly continuous $A$-valued functions on $\R^n$ equipped with the supremum norm. There is an action of $\R^n$ on $C_u(\R^n,A)$ by translation. The subalgebra $\Bi^A(\R^n)$ of smooth elements for the translation action can be deformed in the following way. Let $\Theta$ be a fixed real skew-symmetric $n\times n$ matrix. For $f,g\in\Bi^A(\R^n)$ we define
$$
(f\times_\Theta g)(t):=\iint_{\R^n\times \R^n}f(t+\Theta z)g(t+s)e^{2\pi i z\cdot s}\, dz\, ds,
$$
where the integral has to be understood in the sense of \cite[Prop. 1.6]{Rie1}. Denote by $\Bi_\Theta^A(\R^n)$ the algebra $\Bi^A(\R^n)$ equipped with the new multiplication $\times_\Theta$.

Now let $\Si^A(\R^n)$ be the space of $A$-valued Schwartz functions. We let $f\in\Bi^A_\Theta(\R^n)$ act on $\Si^A(\R^n)$ as
$$
\pi^\Theta(f)g:=f\times_\Theta g,\qquad\forall g\in\Si^A(\R^n).
$$ 
There is an $A$-valued inner product on $\Si^A(\R^n)$, given by
\begin{equation}\label{innerproduct}
\bra f|g\ket_A:=\int_{\R^n}f(s)^*g(s)\, ds, \qquad \forall f,g\in\Si^A(\R^n).
\end{equation}
We denote by $X$ the completion of $\Si^A(\R^n)$ in the norm $\|f\|_A:=\sqrt{\bra f|f\ket_A}$. Then $X$ is a right Hilbert $A$-module and $\pi^\Theta(f)$ is an adjointable operator on $X$ for each $f\in\Bi^A_\Theta(\R^n)$, with adjoint $\pi^\Theta(f)^*=\pi^\Theta(f^*)$ \cite[Prop. 4.2]{Rie1}. Moreover, $\pi^\Theta(f)$ is a bounded operator \cite[Thm. 4.6]{Rie1}.

Write $\Bi_\Theta^A(\R^n)$ for the algebra $\Bi^A(\R^n)$ equipped with the product $\times_\Theta$ and the pre-$C^*$-norm
$$
\|f\|_\Theta:=\|\pi^\Theta(f)\|,\qquad \forall f\in\Bi^A_\Theta(\R^n)
$$
where $\|\cdot\|$ is the operator norm on $\Li_A(X)$. The completion of $\Bi_\Theta^A(\R^n)$ in this norm is a $C^*$-algebra, which we denote by $B_\Theta^A(\R^n)$. Similarly, let $\Si_\Theta^A(\R^n)$ be the algebra $\Si^A(\R^n)$ regarded as a subalgebra of $\Bi^A_\Theta(\R^n)$. Then $\Si_\Theta^A(\R^n)$ is a pre-$C^*$-algebra and in fact \cite[Prop. 3.3]{Rie1} a $*$-ideal in $\Bi^A_\Theta(\R^n)$.  

\begin{dfn}[{\cite[Def. 4.9]{Rie1}}]\label{defofdef}
Suppose that $(A,\R^n,\alpha)$ is a $C^*$-dynamical system and let $\Ai\subset A$ be the subalgebra of smooth elements for the action $\alpha$. For $a\in A$, let $\alpha(a)\in B^A_\Theta(\R^n)$ be the function $\alpha(a)(t):=\alpha_{-t}(a)$. For $a\in\Ai$ we have $\alpha(a)\in\Bi^A_\Theta(\R^n)$. Let $\pi^\Theta:\Ai\to\Li_A(X)$ be the map which takes $a\in\Ai$ to the operator $\pi^\Theta(a)$ given by
$$
(\pi^\Theta(a)g)(t):=(\alpha(a)\times_\Theta g)(t)=\iint_{\R^n\times\R^n}\alpha_{-t+\Theta z}(a)g(t+s)e^{2\pi i z\cdot s}\, dz\, ds
$$
for all $g\in\Si^A(\R^n)$. The \textbf{Rieffel deformation} of $A$ with respect to $(\alpha,\Theta)$ is the $C^*$-algebra $A_\Theta$ obtained by completing $\Ai$ in the norm
$$
\|a\|_\Theta:=\|\pi^\Theta(a)\|,
$$
where $\|\cdot\|$ is the operator norm on $\Li_A(X)$.
\end{dfn}
Thus, $A_\Theta$ is a $C^*$-algebra with multiplication given by
$$
a\times_\Theta b:=\iint_{\R^n\times\R^n}\alpha_{\Theta z}(a)\alpha_s(b)e^{2\pi i z\cdot s}\, dz\, ds
$$
for $a,b$ in the dense subalgebra $\Ai_\Theta$ (we use the subscript $\Theta$ on $\Ai$ when equipped with the product $\times_\Theta$).



\subsubsection{Warped convolutions}
Suppose that $(A,\R^n,\alpha)$ is a $C^*$-dynamical system and that $\pi:A\to\Bi(\Hi)$ is a representation of the $C^*$-algebra $A$. We are now interested in the following task: use $\pi$ to construct an explicit representation of the Rieffel deformation $A_\Theta$ on the same Hilbert space $\Hi$. 

Buchholz, Lechner and Summers introduced a way of deforming an operator $T$ on a Hilbert space $\GH$ to what they called a ``warped convolution" of the operator \cite{BLS}. The idea is as follows. For some positive integer $n$, consider an $n$-tuple of commuting selfadjoint operators $P=(P_\mu)_\mu=(P_0,P_1,\dots,P_{n-1})$ in $\GH$. The notation here is taken from the motivating example of the relativistic momentum operator. There is an associated action 
\begin{equation}\label{Stonesaction}
\alpha_t(T):=e^{it\cdot P}Te^{-it\cdot P}
\end{equation}
of $\R^n$ on $\Bi(\Hi)$. Fix a real antisymmetric $n\times n$ matrix $\Theta$. For a bounded operator $T$ which is smooth with respect to the action \eqref{Stonesaction}, the \textbf{warped convolution} (or just ``warping") of $T$ with respect to $(\alpha,\Theta)$ can be defined as the oscillatory integral
\begin{equation}\label{warped}
T^\Theta:=\int_{\R^n}\alpha_{\Theta s}(T)\, dE^P(s),
\end{equation}
where $dE^P(s)$ is the joint spectral measure of the $P_\mu$'s and $\Theta$ is an $n\times n$ skew-symmetric matrix. In fact, \eqref{warped} makes sense also for certain unbounded operators \cite{Mu1, An1, Mu3}, but we shall only need this fact once (in Proposition \ref{dualgenerators}).

Warped convolution turns out to be related to the deformed products developed by Rieffel. In fact, if $\times_\Theta$ denotes the Rieffel product defined by a unitarly implemented action \eqref{Stonesaction} and the same matrix $\Theta$, then for $\alpha$-smooth operators $S,T\in\Bi(\GH)$ one has \cite{BLS}
\begin{equation}\label{prodrelation}
S^\Theta T^\Theta=(S\times_{\Theta}T)^\Theta.
\end{equation}

\begin{Lemma}[{\cite[Thm. 2.8]{BLS}}]\label{BLSlemma}
Let $(A,\R^n,\alpha)$ be a $C^*$-dynamical system and let $\pi:A\to\Bi(\Hi)$ be a representation in which $\alpha$ is unitarily implemented, i.e. there are selfadjoint operators $D_1,\dots,D_n$ on $\Hi$ such that
$$
\pi(\alpha_t(a))=e^{2\pi it\cdot D}\pi(a)e^{-2\pi it\cdot D},\qquad\forall a\in A,\ t\in\R^n.
$$
 Fix a real skew-symmetric $n\times n$ matrix $\Theta$ and define a map $\pi^\Theta:\Ai\to\Bi(\Hi)$ by
$$
\pi^\Theta(a):=\pi(a)^\Theta,\qquad\forall a\in\Ai,
$$
where $T^\Theta$ is the warped convolution of an operator $T\in\Bi(\Hi)$ with respect to $(\alpha,\Theta)$. Then $\pi^\Theta$ extends to a representation of the Rieffel deformation $A_\Theta$ on $\Hi$. Moreover, $\pi^\Theta$ is faithful iff $\pi$ is faithful, and
$$
\pi^\Theta(\alpha_t(a))=e^{2\pi it\cdot D}\pi^\Theta(a)e^{-2\pi it\cdot D},\qquad\forall a\in A_\Theta,\ t\in\R^n.
$$
\end{Lemma}
In particular, if we have a concrete $C^*$-algebra $A\subset\Bi(\GH)$ equipped with a strongly continuous $\R^n$-action $\alpha$, the $C^*$-algebra generated by $\Ai^\Theta:=\{a^\Theta|\ a\in\Ai\}$ is isomorphic to the Rieffel deformation $A_\Theta$. Here and below, $\Ai\subset A$ denotes the subalgebra of $A$ which is smooth under the action. 

\begin{Example}
Suppose that the $\R^n$-action is periodic, so it can be regarded as an action of the $n$-dimensional torus $\T^n\cong\R^n/\Z^n$. Warped convolution have been used quite a lot in this setting (without identifying it with a warped convolution). If $\alpha$ is a unitarily implemented action on $\Bi(\GH)$ then $\GH$ decomposes into spectral subspaces $\GH^{(r)}$ for $r\in\Z^n$. For a $\T^n$-homogeneous operator $T\in\Bi(\GH)$ of degree $r$, i.e. $\alpha_s(T)=e^{2\pi is\cdot r}T$, the warped convolution of $T$ with respect to $(\alpha,\Theta)$ is the operator $T^\Theta$ which acts as
$$
T^\Theta\xi=e^{2\pi ir\cdot\Theta s}T\xi,\qquad\forall\, \xi\in\GH^{(s)},\ s\in\Z^n;
$$
see \cite[\S 2]{Landi2}, \cite{Yama1}. 
\end{Example}

So let $A\subset\Bi(\GH)$ be a concrete $C^*$-algebra such that the action \eqref{Stonesaction} is strongly continuous on $A$. Then $A_\Theta$ is generated by the operators
\begin{equation}\label{warping}
a^\Theta=\int_{\R^n}e^{i \Theta s\cdot P}a e^{-i\Theta s\cdot  P}\, dE^P(s),\qquad a\in\Ai.
\end{equation}
Whenever we are discussing Rieffel deformations we have a $C^*$-dynamical system $(A,\R^n,\alpha)$. Recall that the $C^*$-algebraic crossed product $B:=A\rtimes_\alpha\R^n$ acts on the Hilbert space $L^2(\R^n,\GH)$ if $A\subset\Bi(\GH)$. Let $t\to\lambda_t=e^{-2\pi it\cdot D}$ be a unitary implementation of $\alpha$ in $L^2(\R^n,\GH)$. Under the embedding $\pi_\alpha:A\to\Mi(B)$ of $A$ into the multiplier algebra of the crossed product $B:=A\rtimes_\alpha\R^n$ we have
$$
\pi_\alpha(\alpha_t(a))=\pi_\alpha(e^{is\cdot P}a e^{-is\cdot  P})=\lambda_t^*\pi_\alpha(a)\lambda_t=e^{2\pi it\cdot D}\pi_\alpha(a)e^{-2\pi it\cdot D}.
$$
Identifying $A_\Theta$ with its concrete image in $\Bi(\GH)$ (the $C^*$-algebra generated by the warpings $a^\Theta$), the $C^*$-dynamical system $(A_\Theta,\R^n,\alpha^\Theta)$ gives rise to a crossed product $B_\Theta:=A_\Theta\rtimes_{\alpha^\Theta}\R^n$ which is represented on the same space $L^2(\R^n,\GH)$ by the map $\pi_{\alpha^\Theta}:A_\Theta\to\Mi(B_\Theta)$. For $T\in A_\Theta$, the operator $\pi_{\alpha^\Theta}(T)$ is given by pointwise multiplication by the operator-valued function $\alpha^\Theta(T)$,
$$
\pi_{\alpha^\Theta}(T)\xi=\alpha^\Theta(T)\xi,\qquad \forall\,\xi\in L^2(\R^n,\GH).
$$
In particular, for a warping $a^\Theta\in\Ai_\Theta$ and a nice vector $\xi\in\Si^A(\R^n,\GH)$ we have
\begin{align*}
(\pi_{\alpha^\Theta}(a^\Theta)\xi)(t)&=(\alpha^\Theta(a^\Theta)\xi)(t)
\\&=(\alpha(a)\times_\Theta \xi)(t)
\\&:=\iint_{\R^n\times\R^n}\alpha_{-t+\Theta z}(a)\xi(t+s)e^{2\pi i z\cdot s}\, dz\, ds
\\&=:(\pi^\Theta(a)\xi)(t),
\end{align*}
so that $\pi_{\alpha^\Theta}(a^\Theta)$ is the operator of ``left Rieffel multiplication" by the function $\alpha(a)$. Recall that  $\pi^\Theta$ appeared also in Definition \ref{defofdef} as a representation of $A_\Theta$ on the Hilbert $A$-module $X$. If we use the identification of the internal tensor product $X\otimes_A\GH$ with $L^2(\R^n,\GH)$ then $\pi_{\alpha^\Theta}$ is the representation on $L^2(\R^n,A)$ induced by $\pi^\Theta:A\to\Li_A(X)$. We will sometimes identify $\pi_{\alpha^\Theta}$ and $\pi^\Theta$ in this way.


\begin{Remark}\label{allinclusive} Note that for $\Theta=0$ (the zero matrix) we have (rewriting the Rieffel product slightly using the Fourier transform)
\begin{align*}
(\pi^\Theta(a)\xi)(t)&=\int_{\R^n}\alpha_{-t+\Theta s}(a)\hat{\xi}(s)e^{2\pi is\cdot t}\, ds
\\&=\alpha_{-t}(a)\int_{\R^n}\hat{\xi}(s)e^{2\pi is\cdot t}\, ds
\\&=\alpha_{-t}(a)\xi(t)
\\&=(\pi_\alpha(a)\xi)(t).
\end{align*}
That is,
$$
\pi^{\Theta\, =\, 0}=\pi_\alpha,
$$
so that when we discuss $\pi^\Theta$ with general skew-symmetric $\Theta$ we automatically include the case $\pi_\alpha$.
\end{Remark}
The operator $\pi^\Theta(a)$ acts by left Rieffel mutliplication with $\alpha(a)$. Since $\pi_{\alpha}(a)$ is the operator of left multiplication with $\alpha(a)$, this means that 
$$
\pi^\Theta(a)=\pi_\alpha(a)^\Theta
$$
is the warped convolution of $\pi_\alpha(a)$ with respect to $(\alpha,\Theta)$, and we have an example of Lemma \ref{BLSlemma} with $\pi=\pi_\alpha$. 
Since $\pi_\alpha(a)$ acts on $L^2(\R^n,\GH)$, the warped convolution \eqref{warping} takes the form
\begin{equation}\label{warpingatltwo}
\pi_\alpha(a)^\Theta=\int_{\R^n}e^{2\pi i \Theta s\cdot D}\pi_\alpha(a)e^{-2\pi i\Theta s\cdot  D}\, dE^D(s).
\end{equation}
We recollect these observations.
\begin{thm}
Let $(A,\R^n,\alpha)$ be a $C^*$-dynamical system and let $\Theta$ be a real skew-symmetrix $n\times n$ matrix. Then the operator $\pi_\alpha(a^\Theta)$ on $L^2(\R^n,\GH)$ is the warped convolution of $\pi_\alpha(a)\in\Ai$ using generators $D$ and matrix $\Theta$. 
\end{thm}
Using $\pi^\Theta(a)=\pi_{\alpha^\Theta}(a^\Theta)=\pi_{\alpha}(a^\Theta)$, we will be able to obtain a formula for the index of operators of the form $\slashed{P}\pi^\Theta(u)\slashed{P}$ in terms of the warpings $u^\Theta\in\Ai_\Theta\subset\Bi(\GH)$. 
 
Index pairings for Rieffel deformations were the original motivation for considering crossed products. The idea was inspired by the third picture of Rieffel deformations, which we recall next.

\subsubsection{Kasprzak deformations}
Let $(A,\R^n,\alpha)$ be a $C^*$-dynamical system. Denote by $\hat{\alpha}:\R^n\to\Aut(B)$ the dual action on the crossed product $B:=A\rtimes_\alpha\R^n$. In Kasprzak's approach to Rieffel deformations, the deforming parameter is (a priori) not a matrix $\Theta$ but a continuous $2$-cocycle 
$$
\Phi:\R^n\times\R^n\to\Un(1)
$$
on the group $\R^n$ with values in the circle group $\Un(1)$. For later comparison we shall label the upcoming deformed objects by $\Theta$ and not by $\Phi$. For each $t\in\R^n$ we have the function $\Phi_t(s):=\Phi(t,s)$ on $\R^n$. Then $\lambda(\Phi_t)$ is an element of $\Mi(B)$, where $\lambda:L^1(\R^n)\to\Mi(B)$ is the embedding. Kasprzak noticed \cite[Thm. 3.1]{Kas} that 
\begin{equation}\label{defdualaction}
\hat{\alpha}^\Theta_t(T):=\lambda(\Phi_t)^*\hat{\alpha}_t(T)\lambda(\Phi_t), \qquad\forall\, T\in B
\end{equation}
defines a strongly continuous action of $\R^n$ on $B$. Moreover, 
$$
\hat{\alpha}^\Theta_t(\lambda_s)=e^{2\pi its}\lambda_s,\qquad \forall\, s,t\in\R^n,
$$
just as the original dual action $\hat{\alpha}$. The idea is now to apply Landstad's theory of crossed products \cite[\S7.8]{Ped1}. 
\begin{dfn}
Let $(A,\R^n,\alpha)$ be a $C^*$-dynamical system and let $\Phi$ be a $2$-cocycle on $\R^n$. The \textbf{Kasprzak deformation}  of $A$ with respect to $(\alpha,\Phi)$ is the Landstad $C^*$-algebra $A_\Theta$ of the $\R^n$-product $(A\rtimes_\alpha\R^n,\lambda,\hat{\alpha}^\Theta)$.
\end{dfn}
Consequently, the Kasprzak deformation $A_\Theta$ satisfies
\begin{equation}\label{crossisom}
A_\Theta\rtimes_{\alpha^\Theta}\R^n\cong A\rtimes_\alpha\R^n
\end{equation}
where $\alpha^\Theta$ is the ``same" action as $\alpha$ but on a different algebra (namely on $A_\Theta$ instead of $A$). 
The algebra $A_\Theta$ was called the ``Rieffel deformation" of $A$ \cite[\S3]{Kas}. 


Kasprzak formulated his deformation for locally compact Abelian groups (not necessarily $\R^n$) \cite{Kas} and his approach extend to not necessarily Abelian groups \cite{BNS} and even to locally compact quantum groups \cite{NT1}.

\subsubsection{Comparison of deformations}\label{deform}
Having described three different ways of deforming a $C^*$-algebra equipped by an $\R^n$-action we now show that it is possible to pass from one to another.

Since Kasprzak used the term ``Rieffel deformation" in his approach, several workers tried to elucidate the relation to Rieffel's deformation by actions of $\R^n$ and the Kasprzak deformation \cite{HM}, \cite{Sang1}, \cite{BNS}, with some success. It was shown in \cite{Ne} that the deformed algebra $A_\Theta$ of Rieffel's satisfies \eqref{crossisom} and is isomorphic to the Kasprzak deformation of $A$ for a canonical choice of $2$-cocycle $\Phi$, whence the notation $A_\Theta$ for both Rieffel and Kasprzak deformations.

That is, if $(A,\R^n,\alpha)$ is a $C^*$-dynamical system and $A_\Theta$ is a Rieffel deformation of $A$ for some choice of matrix $\Theta$, then the result of \cite{Ne} is that the crossed products $B:=A\rtimes_\alpha\R^n$ and $B_\Theta:=A_\Theta\rtimes_\alpha\R^n$ are isomorphic. On the level of smooth crossed products \cite{ENN1}, the explicit isomorphism $\Si^A(\R^n)\ni f \to f^\Theta\in\Si^A_\Theta(\R^n)$ which underlies \eqref{crossisom} is given by \cite{Ne}
\begin{align}\label{neshformula}
f^\Theta(t):=\int_{\R^n}\alpha_{\Theta s}(\hat{f}(s))e^{2\pi it\cdot s}\, ds.
\end{align}
We denote by $\hat{\pi}_\alpha$ and $\hat{\pi}^\Theta$ the representations of $B$ and $B_\Theta$ induced by $\pi_\alpha$ and $\pi^\Theta$ respectively. The important relation is \cite{Ne}
$$
\hat{\pi}^\Theta(f)=\hat{\pi}_\alpha(f^\Theta),\qquad \forall\, f\in\Si^\Ai(\R^n)
$$
where $f$ on the left-hand side is viewed as an element of $B_\Theta$ and on the right-hand side as $f\in B$. 

The notation $f^\Theta$ is used here to stress the similarity with warped convolution. The function $f^\Theta$ defined in \eqref{neshformula} is the Fourier transform of $s\to\alpha_{\Theta s}(\hat{f}(s))$ so in the spectral representation of the $D_k$'s, the operator $\hat{\pi}_{\alpha}(f^\Theta)$ acts as multiplication by the function $s\to\pi_\alpha\big(\alpha_{\Theta s}(\hat{f}(s))\big)$. So 
\begin{align*}
\hat{\pi}_\alpha(f^\Theta)&=\int_{\R^n}\pi_\alpha\big((f^\Theta(t))\big)e^{-2\pi it\cdot D}\, dt
\\&=\iint_{\R^n\times\R^n}e^{2\pi i \Theta s\cdot D}\pi_\alpha\big(\hat{f}(s)\big)e^{-2\pi i\Theta s\cdot  D}e^{2\pi it\cdot s}e^{-2\pi it\cdot D}\, ds\, dt,
\end{align*}
and if $E^D(s)$ is the spectral measure of $D$ then we can write
\begin{align*}
\hat{\pi}_\alpha(f^\Theta)&=\iint_{\R^n\times\R^n}\pi_\alpha\big((f^\Theta(t))\big)e^{-2\pi it\cdot s}\, dE^D(s)\, dt
\\&=\int_{\R^n}e^{2\pi i \Theta s\cdot D}\pi_\alpha\big(\hat{f}(s)\big)e^{-2\pi i\Theta s\cdot  D}\, dE^D(s).
\end{align*}
Recall now \eqref{warpingatltwo}, which says that (under the identification $\alpha^\Theta=\alpha$)
$$
\pi_\alpha(a^\Theta)=\int_{\R^n}e^{2\pi i \Theta s\cdot D}\pi_\alpha(a) e^{-2\pi i\Theta s\cdot  D}\, dE^D(s).
$$
Thus the notion of warped convolution extends to the crossed product by means of the formula \eqref{neshformula}. By considering $\pi_\alpha(A)$ instead of $A$ we can use the isomorphism $B_\Theta\cong B$ etc., and things simplify. The idea is thus to obtained a local formula for Fredholm operators related to the warped convolutions  $\pi_\alpha(a^\Theta)$ by viewing the operator $\pi_\alpha(a^\Theta)$ as a multiplier of the crossed product.  

The relation $\pi_\alpha(a^\Theta)=\pi^\Theta(a)$ is the multiplier analogue of the relation \eqref{neshformula}. Note that this gives
\begin{equation}\label{Neshrelation}
\pi_\alpha(a^\Theta b^\Theta)=\pi_\alpha(a^\Theta)\pi_\alpha(b^\Theta)=\pi^\Theta(a)\pi^\Theta(b)=\pi^\Theta(a\times_\Theta b)=\pi_\alpha((a\times_\Theta b)^\Theta).
\end{equation}

In the following theorem we consider warped convolution with respect to $(\alpha,\Theta)$ for unbounded operators $T$ acting on $L^2(\R^n,\GH)$ and denote by $T^\Theta$ the resulting operator. For the proof, cf. \cite{An1, An2, Mu1}.
\begin{prop}\label{dualgenerators}
Let $X_1,\dots,X_n$ denote the generators of the unitary group implementing the dual action $\hat{\alpha}$ on $B$. Then for all $j\in\{1,\dots,n\}$, as operators with domain $\Si(\R^n;\GH)$ we have
$$
X_j^\Theta=X_j+2\pi \sum^n_{k=1}\Theta_{j,k}D_k.
$$
\end{prop}
We have yet to mention what cocycle should be used in Kasprzak's deformation to obtain the Rieffel deformation $A_\Theta$. If $\Phi:\R^n\times \R^n\to\Un(1)$ is a $2$-cocycle, we let $\Phi_t:\R^n\to\Un(1)$ be the function $\Phi_t(s):=\Phi(t,s)$. As before we have the multiplier $\lambda(\Phi_t)$ of the crossed product $B$. 
\begin{prop}[{\cite[Thm. 2.3]{Ne}}]\label{Kaspcocycl}
Let $(A,\R^n,\alpha)$ be a $C^*$-dynamical system and let $\Theta$ be a real skew-symmetrix $n\times n$ matrix. Let $\Phi:\R^n\times \R^n\to\Un(1)$ be the $2$-cocycle $\Phi(t,s):=e^{-2\pi it\cdot\Theta s}$, so that
$$
\lambda(\Phi_t)=e^{-2\pi i t\cdot \Theta D}.
$$ 
Then the Kasprzak deformation of $A$ by $(\alpha,\Phi)$ is isomorphic to the Rieffel deformation of $A$ by $(\alpha,\Theta)$.
\end{prop}
Combining Proposition \ref{Kaspcocycl} with Proposition \ref{dualgenerators} we conclude yet another relation between warping and Kasprzak deformation.
\begin{cor}\label{corwarpKasprz}
The Kasprzak deformation $A_\Theta$ of $A$ is obtained by replacing the generators $X_1,\dots,X_n$ of the dual action $\hat{\alpha}:\R^n\to\Aut(B)$ by their warped convolutions $X_1^\Theta,\dots,X_n^\Theta$.
\end{cor}
\begin{proof}
The Heisenberg commutation relations $[X_j,D_k]=\sqrt{-1}\delta_{jk}$ imply that the cocycle intertwining the unitary groups generated by $X_j$ and $X_j^\Theta$ is simply $\R^n\ni s\to e^{-is\cdot \Theta D}$. But then the transformation $X\to X^\Theta$ is the Kasprzak approach because changing the dual action as in \eqref{defdualaction} corresponds exactly to the addition of the terms $\Theta_{j,k}D_k$.
\end{proof}

\subsection{Deformed index pairings}
Let $(A,\R^n,\alpha)$ be a $C^*$-dynamical system. For any real skew-symmetric $n\times n$ matrix $\Theta$, the automorphisms $\alpha_t$ act as automorphisms also for the new multiplication on $\Ai_\Theta$ \cite[Prop. 2.5]{Rie1}. As mentioned in Remark \ref{Riefremarks}, $\alpha:\R^n\to\Aut(\Ai_\Theta)$ extends to a strongly continuous action $\alpha^\Theta$ on $A_\Theta$, so we have a new $C^*$-dynamical system $(A_\Theta,\R^n,\alpha^\Theta)$. Moreover, if $\tau$ is an $\alpha$-invariant faithful trace on $A$ then $\tau$ induces an $\alpha^\Theta$-invariant faithful trace on $A_\Theta$ \cite[Thm. 4.1]{Rie2}. We can therefore apply the results of the last section to the deformed system $(A_\Theta,\R^n,\alpha^\Theta)$. 

On the other hand, we need to choose a Hilbert-space representation of $A_\Theta$ in which $\alpha^\Theta$ is unitarily implemented. Warped convolutions allow us to use any representation of the undeformed algebra $A$ in which $\alpha$ is unitarily implemented. In that way, we can use the undeformed embedding $\pi_\alpha:A\to\Mi(B)$ to pass to crossed products. Doing so there might be a chance of obtaining a formula for the index of $\slashed{P}\pi^\Theta(u)\slashed{P}$ in terms of the warped convolution $u^\Theta$. By staying in the original representation $\pi_\alpha$ we could use the relation $\pi^\Theta(a)=\pi_\alpha(a^\Theta)$. This works well, except for the fact that $a^\Theta$ is only defined as a multiplier of $\Si^A(\R^n)$. We will indicate the required modifications in \S\ref{defnumer}.

\subsubsection{The deformed Thom element}
We can use Corollary \ref{corwarpKasprz} to deduce a representative of the Thom element for $(A_\Theta,\R^n,\alpha^\Theta)$ in terms of that of $(A,\R^n,\alpha)$. 

\begin{cor}
The Thom element $\hat{\mathbf{t}}_\alpha^\Theta$ of the dynamical system $(B,\hat{\R}^n,\hat{\alpha}^\Theta)$ is represented by the operator
$$
\slashed{X}^\Theta=\sum^n_{j,k=1}\gamma^k(X_k+2\pi\Theta_{j,k} D_k).
$$
Let $\mathbf{t}_\alpha^\Theta$ be the Thom element for $(A_\Theta,\R^n,\alpha^\Theta)$. Then
$$
\mathbf{t}_\alpha^\Theta\otimes_B\hat{\mathbf{t}}_\alpha^\Theta=1_{A_\Theta}.
$$
\end{cor}
\begin{proof} The first statement comes from Corollary \ref{corwarpKasprz}. The last statement holds because under the isomorphism $B\cong B_\Theta$ induced by $f\to f^\Theta$, the action $\hat{\alpha}^\Theta$ is intertwined with $\hat{\alpha}=\widehat{\alpha^\Theta}$ \cite[Thm. 3.3]{Ne}, and $(B_\Theta,\hat{\R}^n,\widehat{\alpha^\Theta})$ is the ordinary ``dual" Thom element of $(A_\Theta,\R^n,\alpha^\Theta)$.
\end{proof}

\subsubsection{Numerical index for $\Theta\ne 0$}\label{defnumer}
We now try to extend the local index formula from the last chapter to Rieffel deformations. Let $X$ be Hilbert $A$-module obtained by completing $\Si^A(\R^n)$ in the inner product \eqref{innerproduct}.  For nonzero $\Theta$, the best way of taking the trace of the elements $a^\Theta$ seem to be by viewing $a^\Theta$ as an adjointable operator on $X$.

There is a general construction for extending traces to operators on a given Hilbert module (see \cite[\S1]{LN}). In the present case it means that we have to replace $\tau(a^\Theta)$ for $a\in\Ai_+$ by 
\begin{equation}\label{extoftau}
\tilde{\tau}(a^\Theta):=\sup_{\Ii}\sum_{\phi\in\Ii}\tau(\bra\phi|a^\Theta\phi\ket_A),
\end{equation}
where the supremum is taken over all finite subsets $\Ii$ of $X$ for which it holds $\sum_{\phi\in\Ii}\phi\phi^*\leq \bone$, where $\phi\phi^*$ is regarded as a compact operator on $\Si^A(\R^n)$. We denote by $\tilde{\tau}$ this extension of $\tau$ to the $C^*$-algebra $\Li_A(X)$ of adjointable operators on $X$. Note that $\tilde{\tau}$ also extends the trace $\bar{\tau}:A''_+\to[0,+\infty]$.  

In the following we endow the smooth subalgebra $\Ai$ with the Fr\'{e}chet topology given by the seminorms
$$
\|a\|_m:=\sum_{k_1+\cdots+k_n\leq m}\frac{1}{k_1!\cdots k_n!}\|\delta^{k_1}_1\circ\cdots\circ\delta^{k_n}_n(a)\|,\qquad m\in\N_0
$$
where $\delta_1,\dots,\delta_n$ are the generators of the action $\alpha$. 
\begin{Lemma}\label{Riefapproxident} The $\alpha$-smooth subalgebra $\Ai$ has an approximate identity $(e_k)_{k\in\N}$ consisting of positive elements of $A$. Moreover, $(e_k)_{k\in\N}$ is a bounded approximate identity also for the deformed product $\times_\Theta$ on $\Ai$ for any $\Theta$. 
\end{Lemma}
\begin{proof} This is \cite[Props. 2.17, 2.18]{Rie1}.
\end{proof}
\begin{Lemma}\label{multipliertracecalc} Let $\tilde{\tau}$ be the extension of $\tau$ to $\Li_A(X)$ as above. Then for all $a,b\in\Ai_+$ we have
$$
\tilde{\tau}(a^\Theta b^\Theta)=\tau(a\times_\Theta b).
$$
\end{Lemma}
\begin{proof}
Let $(e_k)_{k\in\N}$ be a bounded approximate identity for $\Ai$ and let $(f_k)_{k\in\N}$ be an approximate identity for the convolution algebra $\Si(\R^n)$, where the latter implies that $f_k\geq 0$ and that the Fourier transform of $f_k$ satisfies $\hat{f_k}(0)=1$ for all $k$. Then we get an approximate identity $(\phi_k)_{k\in\N}$ for $\Si^A(\R^n)$ by setting
$$
\phi_k(t):=f_k(t)e_k,\qquad\forall t\in\R^n.
$$
For all $a\in\Ai$, 
\begin{align*}
\lim_k\bra\phi_k|a^\Theta \phi_k\ket&=\lim_k\int_{\R^n}|f_k(t)|^2e_k\, a\times_\Theta e_k\, dt
\\&=\lim_k\widehat{|f_k|^2}(0)e_k(a\times_\Theta e_k)
\\&=a
\end{align*}
where we used Lemma \ref{Riefapproxident} in the last line. Hence we get $a$ back when we apply $\tilde{\tau}$ to $a^\Theta$, even if $\tau$ is not invariant under the $\R^n$-action. On the other hand, for $a,b\in\Ai_+$ we have $a^\Theta b^\Theta=(a\times_\Theta b)^\Theta$.
\end{proof}
From the proof of Lemma \ref{multipliertracecalc} we obtain the following.
\begin{cor} The extension $\tilde{\tau}$ of $\tau$ is finite on every element in the set 
$$
\Dom(\tau)^\Theta:=\{a^\Theta|\ a\in\Dom(\tau)\}.
$$
\end{cor}

\begin{Lemma}\label{deftraceformulalemma} Let $x(t):=ah(t)$ with $h\in L^\infty(\R^n)\cap L^2(\R^n)$ and $\alpha$-smooth $a\in A''$ with $a^*a\in\Dom(\bar{\tau})$. Then $\hat{\pi}^\Theta(x)$ is $\hat{\tau}$-Hilbert-Schmidt and
$$
\hat{\tau}(\hat{\pi}^\Theta(x)^*\hat{\pi}^\Theta(x))=\tau(a^*\times_\Theta a)\int_{\R^n}|h(t)|^2\, dt.
$$
\end{Lemma}
\begin{proof} If $h$ is in $\Si(\R^n)$ then $x$ is in $\Si^\Ai(\R^n)$ and we can define $x^\Theta$ explicitly. This is again an element of $\Si^\Ai(\R^n)$ and hence $\hat{\pi}^\Theta(x)^*\hat{\pi}^\Theta(x)$ is $\hat{\tau}$-traceable. For general $h\in L^\infty(\R^n)\cap L^2(\R^n)$ we have $x\in L^2(\R^n,\Ai)$ and we define $x^\Theta$ by approximation with Schwartz functions. In this way $x^\Theta$ is again in $L^2(\R^n,\Ai)$ and hence $\hat{\pi}_\alpha(x^\Theta)=\hat{\pi}^\Theta(x)$ is Hilbert-Schmidt for $\hat{\tau}$. Proceeding as in Lemma \ref{traceformulalemma} one obtains the formula.
\end{proof}

\begin{Remark}
The action $\alpha$ on $\Si^\Ai(\R^n)$ is just the translation action. From the invariance of the Lebesgue integral under translations one obtains \cite[Prop. 3.6]{Rie1}
$$
\int_{\R^n}(f\times_\Theta g)(t)dt=\int_{\R^n}f(t)g(t)\, dt,
$$
holds for all $f,g\in\Si^\Ai(\R^n)$. Rieffel showed furthermore in \cite[Thm. 4.1]{Rie2} that any $\alpha$-invariant trace $\tau$ on $A$ satisfies
$$
\tau(a\times_\Theta b)=\tau(ab)
$$
for all positive and smooth elements $a,b$. However, this does \emph{not} imply that $\tau(a\times_\Theta b\times_\Theta c)=\tau(abc)$ and so on, and hence the indices in Theorem \ref{headthm} cannot, for $n\geq 2$, be expressed in terms of the undeformed product in general.
\end{Remark}
Our result from the last chapter (Theorem \ref{headthm}) can in particular be applied the deformed system $(A_\Theta,\R^n,\alpha^\Theta)$. We want to rewrite the resulting formula by replacing producs $\times_\Theta$ by the ordinary operator multiplication in $A\subset\Bi(\GH)$. Let $a,b,c\in\Ai$. Using 
$$
a^\Theta b^\Theta c^\Theta=(a\times_\Theta b)^\Theta c^\Theta=(a\times_\Theta b\times_\Theta c)^\Theta
$$
we deduce the relation $\tau(a\times_\Theta b\times_\Theta c)=\tilde{\tau}(a^\Theta b^\Theta c^\Theta)$ as in Lemma \ref{multipliertracecalc}, and similarly for products of $n$ elements in $\Ai$. This gives the following result. 
\begin{thm}\label{headthmRiefl} 
Let $(A,\R^n,\alpha)$ be a $C^*$-dynamical system as in Theorem \ref{headthm} and adapt the notation introduced there. 

If $n$ is odd and $u\in \Ci^\sim$ is unitary for the product $\times_\Theta$, then
\begin{align*}
\Index_{\hat{\tau}}(\slashed{{P}}\pi^\Theta(u)\slashed{{P}})
&=-\hat{\tau}\big(\slashed{ P}[\slashed{ P},\pi_\alpha(u^{\Theta *})][\slashed{ P},\pi_\alpha(u^\Theta)]\cdots[\slashed{ P},\pi_\alpha(u^{\Theta *})][\slashed{ P},\pi_\alpha(u^\Theta)]\big)
\\&=-\frac{2^{(n-1)/2}(-1)^{(n-1)/2}((n-1)/2)!}{(2\pi i)^nn!}\tilde{\tau}\big((u^{\Theta *}\delta(u^\Theta))^n\big),
\end{align*}
where $\tilde{\tau}$ is the extension of $\tau$ to $\Li_A(X)$ defined by \eqref{extoftau}. If $n$ is even then for each projection $e\in\Ci^\sim$ for the product $\times_\Theta$ one has 
\begin{align*}
\Index_{\hat{\tau}}(\boldsymbol\pi^\Theta(e)\slashed{\boldsymbol{R}}_+\boldsymbol\pi^\Theta(e))
&=\frac{1}{2}\hat{\tau}\big(\slashed{\boldsymbol R}\Gamma[\slashed{\boldsymbol R},\boldsymbol\pi_\alpha(e^{\Theta})]\cdots[\slashed{\boldsymbol R},\boldsymbol\pi_\alpha(e^\Theta)]\big)
\\&=\frac{(-1)^{n/2}}{(n/2)!}\frac{2^n}{(2\pi i)^n}\tilde{\tau}\big((e^\Theta\delta(e^\Theta)\delta(e^\Theta))^{n/2}\big).
\end{align*}   
\end{thm}
We stress again that these formulas require that we have a representation (namely the warped convolution) of $A_\Theta$ on the same Hilbert space $\GH$ as $A$ in the first place. Only then do we have $\pi_{\alpha^\Theta}=\pi_\alpha$ as maps from $A_\Theta$ into $\Bi(L^2(\R^n,\GH))$.

\subsection{Some applications}\label{conclsec}
Our original motivation for the present work was to obtain an explicit index pairing for Rieffel deformations. The relevance of such deformations to physics is that they appear when modeling interactions between quantum systems using quantum measurement theory (see \cite{An1, An2}). 


First we give an (counter)example which illustrates the need of an even more general index theory than the one used in this paper.
\begin{Example}[$\kappa$-Minkowski space] The Lebesgue integral $\tau$ defines a trace on the Schwarz algebra $\Si(\R^2)$. A certain star-product $\star_\kappa$ put on a subalgebra $\Ai$ of $\Si(\R^2)$ leads to the noncommutative space called ``$\kappa$-Minkowski space". This can be described as a Rieffel deformation of $\Ai$ \cite[§6]{MS} using an action which does not leave $\tau$ invariant. It is a beautiful fact that $\tau$ is not a trace on $\Ai_\Theta=(\Ai,\star_\kappa)$ but rather a KMS weight with respect to a group of automorphisms of $\Ai$ (we recommend \cite{Ma} for details). The ideas presented in this paper could be a step towards index pairings for $\kappa$-Minkowski. 
\end {Example}
The next example discusses a very well-established application of index pairings in physics, where Rieffel deformations could provide a new tool.
\begin{Example}[Quantum Hall effect]
Consider the commutative $C^*$-algebra $A=C_0(\Omega)$ where $(\Omega,\mu)$ is a probability measure space.  Let $\tau(f):=\int_\Omega f(t)\, d\mu(t)$ be the trace given by integration on $(\Omega,\mu)$ (note that $\tau$ is here finite on all of $A$). Let $\alpha$ be an action of $\R^n$ on $\Ai=C_c^\infty(\Omega)$.  Suppose $X_1,\dots,X_n$ are generators of a unitary group implementing $\alpha$ in $L^2(\R^n\times\Omega,\mu)$. On the crossed product $L^\infty(\Omega,\mu)\rtimes_\alpha\R^n$ there is a weight $\hat{\tau}$ dual to $\tau$ which is a trace if $\tau$ is invariant under $\alpha$; let us assume that this is the case. We let an element $ x\in L^\infty(\Omega,\mu)\rtimes\R^n$ be written formally as 
$$
x\sim \int_{\R^n}x(p)e^{ip\cdot X}
$$
where each $x(p):\Omega\to\C$ belongs to $L^\infty(\Omega,\mu)$. Then the dual weight is given by (see e.g. \cite{Kos})
$$
\hat{\tau}(x)=\tau(x(0))=\int_\Omega (x(0))(\omega)\, d\mu(\omega).
$$
It is also possible to consider a crossed product $\Ni:=L^\infty(\Omega,\mu)\rtimes_{\alpha,\tilde{B}}\R^n$ twisted by a $2$-cocycle of the form $(s,t)\to e^{is\cdot\tilde{B}t}$ on $\R^n$ with a matrix $\tilde{B}$; the dual-trace construction works in this case as well \cite{Su}. The $C^*$-algebra of interest is then the $C^*$-algebraic twisted crossed product $B=C_0(\Omega;\tilde{B})\subset \Ni$. 

Let $\slashed{X}$ denote the Dirac operator formed as in \eqref{Diracoperator} from the generators $X_1,\dots,X_n$ of the unitary group implementing $\alpha$ on $\Ni$. Take $n$ odd, for example, and set
$$
\lambda_n:=-\frac{2^{(n-1)/2}(-1)^{(n-1)/2}((n-1)/2)!}{(2\pi i)^nn!}.
$$
Then for a unitary $u\in\Ai$ we get a formula for the spectral flow $\Sf(\slashed{X},u^*\slashed{X}u)$ from Theorem \ref{headthm},
\begin{align*}
\Sf(\slashed{X},u^*\slashed{X}u)&=\lambda_n\sum_{\epsilon}(-1)^\epsilon\tau\Big(\prod_{k=1}^n u^{-1}\sqrt{-1}[X_{\epsilon(k)},u]\Big)
\\&=\lambda_n\sum_{\epsilon}(-1)^\epsilon\prod_{k=1}^n \int_{\Omega}u^{-1}(\omega)\sqrt{-1}[X_{\epsilon(k)},u](\omega)\, d\mu(\omega),
\end{align*}
the sum running over all permutations $\epsilon$ of $\{1,\dots,n\}$ with sign $(-1)^\epsilon$. According to Theorem \ref{headthmRiefl}, deforming with a matrix $\Theta$ to incorporate the effect of some external interaction, the spectral flow becomes
$$
\Sf^\Theta(\slashed{X},u^*\slashed{X}u)=\lambda_n\sum_{\epsilon}(-1)^\epsilon\prod_{k=1}^n \int_{\Omega}u^{*}(\omega)\times_\Theta\sqrt{-1}[X_{\epsilon(k)},u](\omega)\, d\mu(\omega).
$$
The reader may recognize that what we are discussing here is the setting of the extremely elegant and successful formulation of the integral quantum Hall effect using noncommutative geometry, due to Bellissard et al. \cite{BES}. There, the matrix $\tilde{B}$ which defines the $2$-cocycle $(s,t)\to e^{is\cdot\tilde{B}t}$ is given by $\tilde{B}t:=B\wedge t$, where $B=(B_1,\dots,B_n)$ is the constant magnetic vector field. It has been realized \cite{KeR} that the Bellissard approach is related to the more recent magnetic pseudodifferential calculus of \cite{LM}. Twisted crossed product are very similar to Rieffel deformation but still different \cite{BeM}. Using the results of this paper we can reproduce the quantum Hall algebra and the operators whose Fredholm indices give the quantized conductance, modulo the distinction between crossed products and Rieffel deformation. 

In \cite{BES} the $X_k's$ play the role of position operators, generators of momentum translations, and a spectral triple is defined using $\slashed{X}$. Most prominently, the index of the bounded transform of $\slashed{X}$ (compressed with the Fermi projection) has been used to calculate the Hall conductivity when $n=2$. In \cite{PLB}, this was generalized to the construction of a spectral triple from $\slashed{X}$ to any even $n\geq 1$, and for even $n$ only. The index of Toeplitz operators $PuP$ was used in \cite{ASS} for a mathematical formulation of physical processes, including the integer Hall effect. Theorem \ref{headthm} shows that the $X_k$'s appear also here, and even though the approaches \cite{BES} and \cite{ASS} seem very different at first, the distinction mainly comes from ``even versus odd". For even dimensions the odd pairing can be used by considering unitaries in $C_0(\R,A)^\sim$, which is also what is done in \cite{ASS} for $n=2$ (see also \cite{Go}).
 
There is also a paper \cite{PS} showing the relevance of the Bellissard approach also to odd case. Moreover, results very similar to those of this paper but applicable to twisted crossed products are discussed in \cite{Bour1}. In the same work \cite{Bour1} appears an extensive up-to-date discussion about the $C^*$-approach to topological condensed-matter systems such as the quantum Hall effect.   
\end {Example}

One reason why we are more attracted to the use of Rieffel deformation than twisted crossed products is the direct relation between Rieffel deformation to interactions as they are usually described in quantum physics \cite{An1, An2}. If the $D_k$'s are position operators then the dual action $\hat{\alpha}:\R^n\to\Aut(\Ni)$, implemented by a unitary group $e^{iv\cdot P}$, can be interpreted as the group of spacetime translations. Thus $P=(P_1,\dots,P_n)$ are the energy-momenta. Performing a Rieffel deformation gives that the $P_k$'s are changed by a term coming from the $D_k$'s as we saw in Proposition \ref{dualgenerators}. So the deformation is like adding an external term to the energy or to the momenta, interpreted suitably as coming from the interaction with another quantum system. Note that, by choice of gauge, a transient external electric field can be incorporated either via a potential energy term added to the Hamiltonian, or via an external vector potential term added to $P_1,\dots,P_n$ \cite{BoGK}. We know that either of these can be obtained from Rieffel deformation \cite{An2}.

On the other hand, if we have an action $\alpha$ generated by the momenta $(D_1,\dots,D_n)=(P_1,\dots,P_n)$, then $\slashed{D}=\gamma^kP_k$ is a Dirac operator in the physical sense, and the positive projection $P$ singles out the states of positive energy. The spectral flow between $\slashed{D}$ and $u^*\slashed{D}u$ is then like the amount of charge transferred due to the operation $u$. This could be any real number, although it may be possible to obtain further restrictions on its possible values in specific examples.

Use of such $\slashed{D}$ is not limited to condensed matter physics. In fact, (Lorentzian) spectral triples have been used to define a CAR algebra when the field operators act as multiplication operators with a Moyal product (i.e. the special kind of Rieffel product when the initial algebra $\Ai$ is commutative) \cite{BV, V}. The relevant algebra is thus a Rieffel deformation of a commutative algebra like $\Si(\R^n)$ and the present paper strongly suggests that Connes'-type pairings can be used also in this setting.

\end{document}